\documentclass[a4paper,reqno,11pt]{amsart}
\usepackage[a4paper,margin=1in]{geometry}
\usepackage{amsmath,amssymb,amsthm}

\theoremstyle{plain}
\newtheorem{theorem}{Theorem}[section]

\newtheorem{proposition}[theorem]{Proposition}

\newtheorem{corollary}[theorem]{Corollary}
\theoremstyle{definition}
\newtheorem{remark}[theorem]{Remark}
\numberwithin{equation}{section}

\title[The subharmonic bifurcation of Stokes waves]{The subharmonic bifurcation of Stokes waves on  vorticity flow}

\author{Vladimir Kozlov$^1$}
\address{$^1$Department of Mathematics, Link\"oping University, SE-581 83 Link\"oping, Sweden}

\begin{document}
	
\begin{abstract}

Untill $1980$ one of the main subjects of study in the theory of nonlinear water waves were the Stokes and solitary waves (regular waves).
To that time small amplitude regular waves were constructed and the existence of large amplitude water waves of the same type was proved by using branches of water waves starting
from a trivial (horizontal) wave and ending at extreme waves. Then in papers Chen $\&$ Saffman \cite{Che} and Saffman \cite{Sa} numerical evidence was presented for existence of other type of waves as a result of bifurcations
from a branch of ir-rotational Stokes waves on flow of infinite depth. It was demonstrated that
the Stokes branch has infinitely many bifurcation points when it approaches the extreme wave
and periodic waves with several crests of different height on the period bifurcate from the main
branch. 
The only theoretical works dealing with this phenomenon are Buffoni, Dancer $\&$ Toland \cite{BDT1,BDT2}
where  it was proved the existence of sub-harmonic bifurcations bifurcating from the Stokes branch  for the ir-rotational flow of infinite depth approaching the extreme wave. 

The aim of this paper is to develop new tools and give rigorous
proof of existence of subharmonic bifurcations in the case of rotational flows of finite depth. The whole paper is devoted to the proof of this result formulated in Theorem \ref{Tmain22}.

\end{abstract}

\maketitle

\section{Introduction}

 \subsection{Background}

In this paper we
study two-dimensional flows
of an inviscid, incompressible, heavy fluid, say, water under the
assumption that it is bounded above by a free surface, where the
pressure is constant, and by a horizontal rigid bottom from below. Both irrotational and vortical flows are of
interest and considered here.

It was Stokes \cite{Stokes}, who had initiated mathematical studies of steady
water waves as early as 1847. On the basis of approximations
developed for periodic waves with a single crest per wavelength he
made conjectures about the behaviour of such waves now referred to
as {\it Stokes waves}. These conjectures, to a large extent,
determined the line of investigations of steady waves in the 20th
century (see the paper \cite{PT} by Plotnikov and Toland and references
cited therein). Another line of investigations originated from the
observation made by Scott Russell \cite{Rus} also in the 1840s concerns {\it solitary waves}. Since the 1940s, much
attention was given to proving the existence of such a wave and
investigation of its properties (see, for example, the survey paper
\cite{Gr} by Groves). In particular in Amick $\&$ Toland\cite{AT1,AT2} the existence of global
branches of Stokes/solitary ir-rotational water waves was proved.
Investigation of these classical free-boundary problems requires
most of modern methods in nonlinear functional analysis and
nonlinear dispersive wave theory, in particular, bifurcation theory,
complex variable methods, PDE methods, {\it etc.}

Extreme waves (or, equivalently, “waves of greatest height” according to Stokes) is a unique
phenomenon in the mathematical theory of water waves. These are large-amplitude travelling
waves with sharp crests of included angle $120$ degrees. It is remarkable that while being a
highly nonlinear phenomenon they were predicted by Sir George Stokes already in 1880s, see \cite{Stokes2}.  Stokes also argued that the stagnation by itself
forces the surface profile to have a sharp crest of included angle $2\pi/3$. This property is known
as the Stokes conjecture about waves of greatest height, which stimulated the development of
the theory for many years.

 The first existence of Stokes waves that are arbitrary close to the stagnation is due
to Keady $\&$ Norbury \cite{KNor}, who used a global bifurcation theory for positive operators applied
to the Nekrasov equation. Proof of existence of extreme waves by passing to the limit along
a sequence of waves approaching stagnation was done by Toland \cite{T1a} in 1978 for the
infinite depth case and by Amick and Toland \cite{AT1} for waves of finite depth. The Stokes conjecture concerning the angle $2\pi/3$  was verified independently by Amick, Fraenkel $\&$ Toland in \cite{T2} and by Plotnikov
\cite{P2 }.  Structure of water waves with a stagnation point was studied
by Varvaruca $\&$  Weiss in \cite{VW1}, who proved that if the stagnation point is isolated then the wave  is of class $C^1$ from both sides of it and the angle can be $2\pi/3$ or $\pi$.
All previously mentioned results concerned irrotational water waves, while the case of waves
with vorticity is much less studied. There is also a qualitative difference. In their study
Varvaruca $\&$ Weiss \cite{VW2} found (without proving the existence) that surface profiles near stagnation
points are either Stokes corners, horizontally flat, or horizontal cusps, though it is not
known if the last two options are possible. It was shown in Varvaruca
\cite{Varv} that there exists a family of periodic solutions to the water wave problem with ”negative”
vorticity converging to an extreme wave enjoying stagnation at every crest. Unfortunately, it
was not possible to show that the limiting wave is not ”trivial”, that is not a laminar flow
with surface stagnation. This difficulty was resolved in Kozlov $\&$ Lokharu \cite{KL1} by using a different approach and
extreme waves  were found. A further analysis was made in Kozlov $\&$ Lokharu \cite{KL4,KL3}, where authors
obtained higher-order asymptotics for the surface profile near stagnation points. We note also the papers \cite{P1} by Plotnikov, where a global branches of solitary waves in ir-rotational  case was constructed and bifurcation points were studied.

 Until
2000, only a few works treated steady water waves with vorticity, but activity in this area
became very intensive during the last twenty years; see the survey \cite{Str} by Strauss, which covers works
that had been published in the field of steady waves during the past nine decades. This review
contains almost 100 references on both irrotational and vortical waves. Existence of small amplitude  Stokes waves is established in \cite{CSst} (unidirectional water waves), \cite{KN14} (water waves with counter-currents)  and existence of
solitary waves for near-critical values of Bernoulli's constant is proved in Groves $\&$ Wahlen \cite{GrW} and Hur \cite{Hur}. We mention here the
construction of a global branch of Stokes waves without counter-currents in Constantin $\&$ Strauss \cite{CSst,C8} and a global branch of solitary waves in Chen, Walsh $\&$ Wheeler \cite{CWW}.

Stokes and solitary waves (regular waves) were the main subject of study up to 1980. In 1980 (see Chen \cite{Che} and Saffman \cite{Sa}) it was discovered numerically and in 2000 (see \cite{BDT1,BDT2}) this was supported theoretically for the ir-rotational case for flow of infinite depth that there exist new types of periodic waves with several crests on the period (the Stokes wave has only one crest). This waves appear as a result of bifurcation from a branch of Stokes waves when they approach the wave of greatest amplitude. The only theoretical investigation is \cite{BDT1,BDT2} which was devoted to subharmonic bifurcations from a branch of Stokes waves on a an ir-rotational flow of infinite depth.

\vspace{4mm}
\subsection{The method}


 Starting point of our study is a branch of Stokes waves approaching an extreme wave. This is based on our papers
\cite{KL1} and \cite{KL3}, where it is proved existence of branches of Stokes waves approaching an extreme wave in the rotational case and for a  flow of finite depth. We choose the period of Stokes wave as a parameter for constructed branches. This is important to guarantee that an extreme periodic wave will appear as the limit configuration for the branch. 

 In further analysis we need certain smoothness properties of extreme waves, which was the subject of the study in the paper  \cite{KL4} (see also a short presentation of the results in \cite{KL3}).  In previous papers (see Varvaruca \cite{Varv}, Varvaruca $\&$ Weiss \cite{VW1,VW2}) it was prove that the extreme wave is $C^1$ from both sides of the stagnation point. We find an optimal asymptotics for the limit wave near the stagnation in \cite{KL4}. Such properties of the extreme wave are important in order to show that there are infinitely many bifurcations at the branch of Stokes waves approaching an extreme wave.


The above results are  first steps in our study of subharmonic bifurcation. Our study of bifurcations is based on the variational structure of operators. Earlier such structure was used by Plotnikov in \cite{P1}, where he investigated bifurcations of solitary waves, by Buffoni, Dancer $\&$  Toland in \cite{BDT1,BDT2} and by Shargorodsky $\&$ Toland in \cite{ShTo}, devoted to irrotational Stokes waves.  Since our family of operators consists of potential operators we apply the bifurcation theorem for potential family of operators whose coefficients comes from the branch of Stokes waves connecting the stream solution and extreme wave. In order to find subharmonic bifurcations we are looking for solutions with period multiple of the period of coefficient of the main branch of Stokes waves. The principle difficulties here is to separate bifurcation points corresponding to the Stokes waves and to the waves with multiple period (subharmonic waves) and to study the change of the Morse index. This is done by using the results from \cite{KL2}  and new study of negative spectrum of the first variation of water waves close to the extreme wave presented here. Two properties of negative spectrum are important: when we approaching the extreme wave the negative eigenvalues can be arbitrary large in absolute value, the corresponding eigenfunctions are concentrated neat the stagnation point.

\subsection{Formulation of the problem}

We consider steady surface waves in a two-dimensional channel bounded below by a flat,
rigid bottom and above by a free surface that does not touch the bottom. The surface tension is neglected and the water motion can be rotational.
In appropriate Cartesian coordinates $(X, Y )$, the bottom coincides with the
$X$-axis and gravity acts in the negative $Y$ -direction. We choose the frame of reference so that the velocity field is time-independent as well as the free-surface profile
which is supposed to be the graph of $Y = \xi(X)$, $X \in \Bbb R$, where $\xi$ is a positive and
continuous unknown function. Thus
$$
{\mathcal D}_\xi = \{X\in \Bbb R, 0 < Y < \xi(X)\},\;\;{\mathcal B}_\xi=\{X\in\Bbb R,\;Y=\xi(X)\}
$$
is the water domain and the free surface respectively. We will use the stream function $\Psi$, which is connected with the velocity vector $({\bf u},{\bf v})$ as ${\bf u}=-\Psi_Y$ and ${\bf v}=\Psi_X$.

We assume that $\xi$ is a positive, periodic function having period $\Lambda>0$ and that $\xi$   is even and strongly monotonically decreasing on the interval $(0,\Lambda/2)$.  Since the surface tension is neglected, $\Psi$ and
$\xi$ after a certain scaling satisfy the following free-boundary problem (see for example \cite{KN14}):
\begin{eqnarray}\label{K2a}
&&\Delta \Psi+\omega(\Psi)=0\;\;\mbox{in ${\mathcal D}_\xi$},\nonumber\\
&&\frac{1}{2}|\nabla\Psi|^2+\xi=R\;\;\mbox{on ${\mathcal B}_\xi$},\nonumber\\
&&\Psi=1\;\;\mbox{on ${\mathcal B}_\xi$},\nonumber\\
&&\Psi=0\;\;\mbox{for $Y=0$},
\end{eqnarray}
where $\omega\in C^{1,\gamma}$, $\gamma\in (0,1)$, is a vorticity function and $R$ is the Bernoulli constant. The Frechet derivative of the non-linear operator is evaluated in \cite{KL2} (see also Sect. \ref{SOkt10a} in this paper) and it is represented by the left-hand side in the following spectral problem
\begin{eqnarray}\label{K2az}
&&\Delta w+\omega'(\Psi)w=-\frac{\theta}{\Psi_Y} w\;\;\mbox{in ${\mathcal D}_\xi$},\nonumber\\
&&\partial_\nu w-\rho w=0\;\;\mbox{on ${\mathcal B}_\xi$},\nonumber\\
&&\Psi=0\;\;\mbox{for $Y=0$},
\end{eqnarray}
where $\theta$ is a real number,
$$
\rho(X,Y)=\frac{1+\Psi_{X}\Psi_{XY}+\Psi_{Y}\Psi_{YY}}{\Psi_Y(\Psi_X^2+\Psi_Y^2)^{1/2}}
$$
and $w$ is an even, $\Lambda$-periodic function.

In what follows we will study  branches of Stokes waves $(\Psi(X,Y;t),\xi(X;t))$ of period $\Lambda(t)$, $t\geq 0$, started from the uniform stream  at $t=0$.
It is convenient to make the following change of variables
\begin{equation}\label{D11a}
x=\lambda X,\;\;y=Y,\;\;\lambda=\frac{\Lambda(0)}{\Lambda(t)}
\end{equation}
in order to deal with the problem with a fixed period. As the result we get
\begin{eqnarray}\label{Okt6aa}
&&\Big(\lambda^2\partial_x^2+\partial_y^2\Big)\psi+\omega(\psi)=0\;\;\mbox{in $D_\eta$},\nonumber\\
&&\frac{1}{2}\Big(\lambda^2\psi_x^2+\psi_y^2\Big)+\eta=R\;\;\mbox{on $B_\eta$},\nonumber\\
&&\psi=1\;\;\mbox{on $B_\eta$},\nonumber\\
&&\psi=0\;\;\mbox{for $y=0$},
\end{eqnarray}
where
$$
\psi(x,y;t)=\Psi(\lambda^{-1}x,y;t)\;\;\mbox{and}\;\;\eta(x;t)=\xi(\lambda^{-1} x;t).
$$
Here all functions have the  same period $\Lambda_0:=\Lambda(0)$, $D_\eta$ and $B_\eta$  are the domain and the free surface  after the change of variables (\ref{D11a}).


\subsection{Uniform stream solution, dispersion equation}

The uniform stream solution $\psi=U(y)$ with the constant depth $\eta =d$ and $\lambda=1$,  satisfies the problem
\begin{eqnarray}\label{X1}
&&U^{''}+\omega(U)=0\;\;\mbox{on $(0;d)$},\nonumber\\
&&U(0)=0,\;\;U(d)=1,\nonumber\\
&&\frac{1}{2}U'(d)^2+d=R.
\end{eqnarray}
Let $s=U'(0)$ and $s>s_0:=2\max_{\tau\in [0,1]}\wp(\tau)$, where
$$
\wp(\tau)=\int_0^\tau \omega(p)dp.
$$
Then the problem (\ref{X1}) has a solution $(U,d)$ with a strongly monotone function $U$ if
\begin{equation}\label{M6a}
\frac{1}{2}s^2+d(s)-\wp(1)=R.
\end{equation}
In this case $(U,d)$ is found from the relations
\begin{equation}\label{F22a}
y=\int_0^U\frac{d\tau}{\sqrt{s^2-2\wp(\tau)}},\;\;d=d(s)=\int_0^1\frac{d\tau}{\sqrt{s^2-2\wp(\tau)}}.
\end{equation}
The equation (\ref{M6a}) is solvable if $R>R_c$,
\begin{equation}\label{F27a}
R_c=\min_{s\geq s_0}\Big(\frac{1}{2}s^2+d(s)-\wp(1)\Big),
\end{equation}
and it has two solutions if
$R\in (R_c,R_0)$, where
\begin{equation}\label{D19ba}
R_0=\frac{1}{2}s_0^2+d(s_0)-\wp(1).
\end{equation}
We denote by $s_c$ the point where the minimum in (\ref{F27a}) is attained.

Existence of small amplitude Stokes waves is determined by the dispersion equation (see, for example, \cite{KN14}). It is defined as follows.
 The strong monotonicity of $U$ guarantees that the problem
\begin{equation}\label{Okt6b}
\gamma^{''}+\omega'(U)\gamma-\tau^2\gamma=0,\;\; \gamma(0,\tau)=0,\;\;\gamma(d,\tau)=1
\end{equation}
has a unique solution $\gamma=\gamma(y,\tau)$ for each $\tau\in\Bbb R$, which is even with respect to $\tau$ and depends analytically on $\tau$.
Introduce the function
\begin{equation}\label{Okt6ba}
\sigma(\tau)=\kappa\gamma'(d,\tau)-\kappa^{-1}+\omega(1),\;\;\kappa=\Psi'(d).
\end{equation}
It depends also analytically on $\tau$ and it is strongly increasing with respect to $\tau>0$. Moreover it is an even function.
The dispersion  equation (see, for example \cite{KN14})  is the following
\begin{equation}\label{Okt6bb}
\sigma(\tau)=0.
\end{equation}
It has a positive solution if
\begin{equation}\label{D17a}
\sigma(0)<0.
\end{equation}
By \cite{KN14} this is equivalent to $s+d'(s)<0$ or what is the same
\begin{equation}\label{D19b}
1<\int_0^d\frac{dy}{U'^2(y)}.
\end{equation}
The left-hand side here is equal to $1/F^2$ where $F$ is the Froude number (see \cite{We} and \cite{KLW}). Therefore (\ref{D19b}) means that $F<1$, which is well-known condition for existence of  water waves of small amplitude.
Another equivalent formulation is given by requirement (see, for example \cite{KN11})
\begin{equation}\label{M3aa}
s\in (s_0,s_c)\;\;\mbox{and satisfies (\ref{M6a})}.
\end{equation}
The existence of such $s$ is guaranteed by $R\in (R_c,R_0)$. 

The function $\sigma$ has the following asymptotic representation
$$
\sigma(\tau)=\kappa\tau +O(1)\;\;\mbox{for large $\tau$}
$$
and equation (\ref{Okt6bb}) has a unique positive root, which will be denoted by $\tau_*$. It is connected with $\Lambda_0$ by the relation
$
\tau_*=\frac{2\pi}{\Lambda_0}.
$
The function $\gamma (y,\tau)$ is positive in $(0,d]$ for $\tau>\tau_*$.

Let
\begin{equation}\label{F28a}
\rho_0=\frac{1+\Psi'(d)\Psi^{''}(d)}{\Psi'(d)^2}.
\end{equation}
We note that
$$
\frac{1+\Psi'(d)\Psi^{''}(d)}{\Psi'(d)^2}=\kappa^{-2}-\frac{\omega(1)}{\kappa}
$$
and hence another form for (\ref{Okt6ba}) is
\begin{equation}\label{M21aa}
\sigma(\tau)=\kappa\gamma'(d,\tau)-\kappa\rho_0.
\end{equation}

\subsection{Partial hodograph transform}

When we consider a branch of solutions to the problem (\ref{Okt6aa}) it is convenient to work with a domain independent of $t$. One set to do this is to use the partial hodograph transform, see \cite{CSst} or \cite{KN14}.


We assume that
$$
\psi_y>0\;\;\mbox{in $\overline{D_\eta}$}
$$
and use the variables
$$
q=x,\;\;p=\psi.
$$
Then
$$
q_x=1,\;\;q_y=0,\;\;p_x=\psi_x,\;\;p_y=\psi_y,
$$
and
\begin{equation}\label{F28b}
\psi_x=-\frac{h_q}{h_p},\;\;\psi_y=\frac{1}{h_p},\;\;dxdy=h_pdqdp.
\end{equation}

System (\ref{Okt6aa}) in the new variables takes the form
\begin{eqnarray}\label{J4a}
&&\Big(\frac{1+\lambda^2h_q^2}{2h_p^2}+\wp(p)\Big)_p-\lambda^2\Big(\frac{h_q}{h_p}\Big)_q=0\;\;\mbox{in $Q$},\nonumber\\
&&\frac{1+\lambda^2h_q^2}{2h_p^2}+h=R\;\;\mbox{for $p=1$},\nonumber\\
&&h=0\;\;\mbox{for $p=0$}.
\end{eqnarray}
Here
$$
Q=\{(q,p)\,:\,q\in\Bbb R\,,\;\;p\in (0,1)\}.
$$
The uniform stream solution corresponding to the solution $U$ of (\ref{X1}) is
\begin{equation}\label{M4c}
H(p)=\int_0^p\frac{d\tau}{\sqrt{s^2-2\wp(\tau)}},\;\;s=U'(0)=H_p^{-1}(0).
\end{equation}
One can check that
$$
H_{pp}-H_p^3\omega(p)=0
$$
or equivalently
$$
\Big(\frac{1}{2H_p^2}\Big)_p+\omega(p)=0.
$$
Moreover it satisfies the boundary conditions
\begin{equation}\label{M4ca}
\frac{1}{2H_p^2(1)}+H(1)=R,\;\;H(0)=0.
\end{equation}
The problem (\ref{J4a}) has a variational formulation (see \cite{CSS}) and the potential is given by
\begin{equation}\label{M2a}
f(h;\lambda)=\int_{-\Lambda_0/2}^{\Lambda_0/2}\int_0^1\Big (\frac{1+\lambda^2h_q^2}{2h_p^2}-h+R-(\wp(p)-\wp(1))\Big)h_pdqdp.
\end{equation}

We introduce spaces $C^{k,\alpha}_{pe}(\overline{Q})$ and $C^{k,\alpha}_{pe}(\Bbb R)$, $k=0,1,2,\ldots$, $\alpha\in (0,1)$, which are subspaces of $C^{k,\alpha}(\overline{Q})$ and $C^{k,\alpha}(\Bbb R)$ respectively and consist of even, periodic functions of the period $\Lambda_0$.

The following theorem is proved in \cite{KL1}.


\begin{theorem}\label{T11} Assume that $\omega\in  C^{1,\gamma}([0;1])$. Then for any $R\in (R_c, R_0)$
there exist functions ${\mathcal C} : [0;+\infty)\to C^{2,\gamma}_{pe}(\overline{Q})$ and $ \lambda : [0;+\infty) \to (0;+\infty)$ solving the problem (\ref{J4a}) for each $t\geq 0$ with the following properties:

{\rm (i)} ${\mathcal C}(0)$ is the subcritical stream solution $H(p)$ given by (\ref{M4c});

{\rm (ii)} ${\mathcal C}(t) = h(q, p; t); t > 0$ is a solution to (\ref{J4a}) with $\lambda = \lambda (t)$; the corresponding solution
$(\psi; \eta )$ represents a Stokes wave with the period $\Lambda(t)  = \Lambda_0\lambda(t)^{-1}$;

{\rm (iii)} $({\mathcal C}(t),\lambda(t))$ has a real analytic reparametrization locally around each $t\geq 0$.

Furthermore, there exists a sequence $\{t_j\}_{j=1}^\infty$, $t_j\to\infty$ as $j\to\infty$, such that we either have

{\rm (I)} $\max_{x\in\Bbb R}\eta(x;t_j)\rightarrow R$  and $\Lambda(t_j)\rightarrow\Lambda < \infty$, $\Lambda\neq 0$, as $j\rightarrow\infty$ (stagnation at every crest);

{\rm (II)} $\lim_{j\to\infty}\max_{x\in\Bbb R}\eta(x;t_j)=R_1\leq R$ and $\Lambda(t_j)\rightarrow\infty$ as $j\to \infty$ (solitary wave, possibly with stagnation).

{\rm (III)}  $\lim_{j\to\infty}\max_{x\in\Bbb R}\eta'(x;t_j)=\infty$ as $j\to\infty$.

Here $\eta (x;t)$ is the surface profile corresponding to $h(q,p;t)$, i.e. $\eta(x;t)=h(x,1;t)$.

\end{theorem}

\subsection{Assumptions}\label{S24a}

Here we collect basic assumptions which will be used in this paper for proving our main result on subharmonic bifurcations.

 We assume that $R\in (R_c,R_0)$ and (\ref{D19b}) holds. The last assumption is equivalent to (\ref{D17a}) and to (\ref{M3aa}). Then according to Theorem \ref{T11} there exists a branch of solutions to (\ref{J4a})
\begin{equation}\label{J4ac}
h=h(q,p;t):[0,\infty)\rightarrow C^{2,\gamma}_{pe}(\overline{Q}),\;\;\lambda=\lambda(t):[0,\infty)\rightarrow (0,\infty),
\end{equation}
which  has a real analytic reparametrization locally around each $t\geq 0$.

We assume that
\begin{equation}\label{D21ab}
|h_q(q,1;t)|\leq {\mathcal B} \;\;\mbox{for $q\in\Bbb R$ and $t\geq 0$},
\end{equation}
where ${\mathcal B}$ is a positive constant and that the branch satisfies  Theorem \ref{T11} (I), i.e there exists a sequence $\{t_j\}_{j=1}^\infty$ such that $t_j\to \infty$ as $j\to \infty$ and
\begin{equation}\label{M29a}
h(0,1;t_j)\to R,\;\;\mbox{}\;\;\lambda(t_j)\to \lambda_*\;\;\mbox{as $j\to\infty$},
\end{equation}
where $\lambda_*$ is a positive constant.

We introduce the sequence $(\Psi_j(X,Y),\xi_j(X), \Lambda_j)$, which corresponds to $h_j=h(q,p;t_j)$ and solves the problem (\ref{K2a}). In particular,
\begin{equation}\label{A27a}
\xi_j(X)=h_j(\lambda_jX,0),\;\;\lambda_j=\frac{\Lambda_0}{\Lambda_j}
\end{equation}
and $\Psi_j$ can be found from
\begin{eqnarray}\label{A27aa}
&&\Delta \Psi_j+\omega(\Psi_j)=0\;\;\mbox{in ${\mathcal D}_{\xi_j}$}\nonumber\\
&&\Psi_j=1\;\;\mbox{on ${\mathcal B}_{\xi_j}$}\nonumber\\
&&\Psi_j=0\;\;\mbox{for $Y=0$}.
\end{eqnarray}
The remaining boundary condition for $\Psi_j$ holds because of a similar equation for $h_j$. Both functions $\xi_j$ and $\Psi_j$ have period $\Lambda_j$ with respect to $X$. Furthermore the consequence of (\ref{D21ab}) is the inequality
\begin{equation}\label{Ma8a}
|\xi_{jX}(X)|\leq\lambda_j{\mathcal B}.
\end{equation}

By Proposition 4.2, \cite{KLN17},
\begin{equation}\label{D21a}
|\nabla\Psi_j(X,Y)|\leq C(R,\omega_0),\;\;\mbox{for $(X,Y)\in \overline{{\mathcal D}_{\xi_j}}$},
\end{equation}
where $C$ depends only on $R$ and $\omega_0$, $\omega_0=\max_{0\leq p\leq 1}\omega(p)$. Due to (\ref{D21ab}) one can choose a subsequence of $\{\xi_j\}$ which is convergent in $C^{0,\alpha}_{pe}(\Bbb R)$ for any $\alpha\in (0,1)$, to a function $\xi_*\in C^{0,1}_{pe}(\Bbb R)$. By  (\ref{D21a}) we can assume also that the sequence $\{\widetilde{\Psi}_j\}$, where $\widetilde{\Psi}_j$ is the extension of $\Psi_j$ by $1$ for $Y>\xi(X)$ and by $0$ for $Y<0$,  is also convergent in $L^\infty (Q_{R,a,b})$ to a function $\widetilde{\Psi}_*$, where
$$
Q_{R,a,b}=\{(X,Y)\,:\,X\in (a,b),\;Y\in (0,R)\}.
$$
Moreover
$$
\nabla\widetilde{\Psi}_j\; \mbox{weak}^*\;\nabla\widetilde{Psi_*}\;\mbox{in}\; L^\infty (Q_{R,a,b}).
$$
Here $a<b$ are arbitrary real numbers.
Then the limit functions $(\xi_*,\Psi_*,\Lambda_*)$, where $\Lambda_*=\Lambda_0/\lambda_*$ and $\xi_*$ together with $\Psi_*$ are even periodic functions of period $\Lambda_*$, satisfying (\ref{K2a}) in a weak sense, see \cite{Varv}. One can verify that all conditions of Theorem 5.2, Varvaruca \cite{Varv}, are satisfied and
according to that theorem  there are two options for the limit function $\xi_*$:
\begin{equation}\label{F24a}
\lim_{X\to 0+}\frac{\xi_*(X)}{X}=-\frac{1}{\sqrt{3}}\;\;\mbox{or}\;\;\lim_{X\to 0}\frac{\xi_*(X)}{X}=0.
\end{equation}
 We assume that  the first option holds, i.e. the limit Stokes wave has an opening angle  $120$ degree at the stagnation points.

\bigskip
Below we give  sufficient conditions for validity of the above assumptions.

\begin{proposition}\label{Pm22} (Existence of highest waves, Kozlov $\&$ Lokharu \cite{KL3}) Let $\omega\geq 0$ be sufficiently small. Then there exists a constant $R_*\in (R_c, R_0)$  such that only waves of type (I) with the opening angle $120$ degree  can occur for $R\in (R_*,R_0)$.
\end{proposition}

\begin{remark} Validity of (\ref{D21ab}) is analysed in
Strauss $\&$ Wheeler \cite{STrWh}.
\end{remark}

\section{Bifurcation analysis}

Here we present two equation for finding bifurcation points in $q,p$ variables. One of them is defined by a boundary value problem in a two-dimensional domain and another one including a Dirichlet-Neumann operator is defined one a part of one-dimensional boundary. It was proved in Sect. \ref{SOkt10b} that the Frechet derivatives of operators in these two formulations have the same number of negative eigenvalues and the same dimension of the kernels. So both of them can be used in analysis of the Morse index and corresponding crossing number. The first bifurcation $2$D problem is useful in analysis of the number of negative eigenvalues and the second one is more convenient for application of general bifurcation results.

Both formulations can be easily extended for study of subharmonic bifurcations, see Sect. \ref{SA22a} In Sect. \ref{SOkt10a} we give an explicit connection between the Frechet derivatives in $(x,y)$ and $(q,p)$ variables.


\subsection{First formulation of bifurcation equation}\label{SA23a}

In order to find bifurcation points and bifuracating solutions we put $h+w$ instead of $h$ in (\ref{J4a}) and introduce the operators
\begin{eqnarray*}
&&{\mathcal F}(w;h,t)=\Big(\frac{1+\lambda^2(h_q+w_q)^2}{2(h_p+w_p)^2}\Big)_p
-\Big(\frac{1+\lambda^2h_q^2}{2h_p^2}\Big)_p\\
&&-\lambda^2\Big(\frac{h_q+w_q}{h_p+w_p}\Big)_q+\lambda^2\Big(\frac{h_q}{h_p}\Big)_q
\end{eqnarray*}
and
$$
{\mathcal G}(w;h,t)=\frac{1+\lambda^2(h_q+w_q)^2}{2(h_p+w_p)^2}-\frac{1+\lambda^2h_q^2}{2h_p^2}+w
$$
acting on $\Lambda_0$-periodic functions $w$ defined in $Q$. After some cancelations we get
$$
{\mathcal F}=\Big(\frac{\lambda^2h_p^2(2h_q+w_q)w_q-(2h_p+w_p)(1+\lambda^2h_q^2)w_p}{2h_p^2(h_p+w_p)^2}\Big)_p
-\lambda^2\Big(\frac{h_pw_q-h_qw_p}{h_p(h_p+w_p)}\Big)_q
$$
and
$$
{\mathcal G}=\frac{\lambda^2h_p^2(2h_q+w_q)w_q-(2h_p+w_p)(1+\lambda^2h_q^2)w_p}{2h_p^2(h_p+w_p)^2}+w.
$$
Both these functions are well defined  for small $w_p$.
Then the problem for finding solutions close to $h$ is the following
\begin{eqnarray}\label{F19a}
&&{\mathcal F}(w;h,t)=0\;\;\mbox{in $Q$}\nonumber\\
&&{\mathcal G}(w;h,t)=0\;\;\mbox{for $p=1$}\nonumber\\
&&w=0\;\;\mbox{for $p=0$}.
\end{eqnarray}
The above problem  also has a variational formulation. The corresponding potential is
$$
f(h+w;\lambda)-f(h;\lambda),
$$
where $f$ is defined by (\ref{M2a}).
Furthermore, the Frechet derivative (the linear approximation of the functions ${\mathcal F}$ and ${\mathcal G}$) is the following
\begin{equation}\label{J4aa}
Aw=A(t)w=\Big(\frac{\lambda^2h_qw_q}{h_p^2}-\frac{(1+\lambda^2h_q^2)w_p}{h_p^3}\Big)_p-\lambda^2\Big(\frac{w_q}{h_p}-\frac{h_qw_p}{h_p^2}\Big)_q
\end{equation}
and
\begin{equation}\label{J4aba}
{\mathcal N}w={\mathcal N}(t)w=(N w-w)|_{p=1},
\end{equation}
where
\begin{equation}\label{J4ab}
N w=N(t)w=\Big(-\frac{\lambda^2h_qw_q}{h_p^2}+\frac{(1+\lambda^2h_q^2)w_p}{h_p^3}\Big)\Big|_{p=1}.
\end{equation}
The eigenvalue problem for the Frechet derivative, which is important for the analysis of bifurcations of the problem
(\ref{F19a}), is the following
\begin{eqnarray}\label{M1a}
&&A(t)w=\theta w\;\;\mbox{in $Q$},\nonumber\\
&&{\mathcal N}(t)w=0\;\;\mbox{for $p=1$},\nonumber\\
&&w=0\;\;\mbox{for $p=0$}.
\end{eqnarray}

\begin{remark}\label{RA9} Differentiating equations in (\ref{J4a}) with respect to $q$ we get
\begin{eqnarray}\label{A9a}
&&A(t)h_q=0\;\;\mbox{in $Q$},\nonumber\\
&&{\mathcal N}(t)h_q=0\;\;\mbox{for $p=1$},\nonumber\\
&&h_q=0\;\;\mbox{for $p=0$}.
\end{eqnarray}
Therefore $h_q$ always solves the problem (\ref{M1a}) for $\theta=0$ but we note that the function $h_q$ is odd with respect to $q$.
\end{remark}

\subsection{Second formulation of bifurcation equation}\label{SA23b}

There is another formulation for the bifurcating solutions, which is used the Dirichlet-Neumann operator. Let us consider the problem
\begin{eqnarray}\label{F19aa}
&&{\mathcal F}(w;h,t)=0\;\;\mbox{in $Q$},\nonumber\\
&&w=g\;\;\mbox{for $p=1$},\nonumber\\
&&w=0\;\;\mbox{for $p=0$}.
\end{eqnarray}
We define the operator ${\mathcal S}={\mathcal S}(g;h,t)$ by
\begin{equation}\label{F19ab}
{\mathcal S}(g;t)={\mathcal G}(w;h,t)|_{p=1},
\end{equation}
where $w$ is the solution of the problem (\ref{F19aa}). Then the equation for bifurcating solutions is
\begin{equation}\label{F19b}
{\mathcal S}(g;t)=0.
\end{equation}
We will prove the solvability of the problem (\ref{F19aa}) in Sect.\ref{F19aa}. Here we note that spectral problem for the the Frechet derivative of the left-hand side in (\ref{F19b}) is given by
\begin{eqnarray}\label{F19ba}
&&A(t)w=0\;\;\mbox{in $Q$},\nonumber\\
&&N(t)w-w=\mu w \;\;\mbox{for $p=1$},\nonumber\\
&&w=0\;\;\mbox{for $p=0$}.
\end{eqnarray}
More exactly if we introduce the problem
\begin{eqnarray}\label{F19bb}
&&A(t)w=0\;\;\mbox{in $Q$},\nonumber\\
&&w=g \;\;\mbox{for $p=1$},\nonumber\\
&&w=0\;\;\mbox{for $p=0$},
\end{eqnarray}
then the operator
\begin{equation}\label{M3a}
Sg=(Nw-w)|_{p=1}
\end{equation}
is the Frechet derivative of the operator  (\ref{F19b}). The corresponding spectral problem is
\begin{equation}\label{M3b}
Sg=\mu g.
\end{equation}

\subsection{Spectral problems (\ref{M1a}) and (\ref{F19ba}) in $(x,y)$ variables}\label{SOkt10a}

In this section we write the spectral problems (\ref{M1a}) and (\ref{F19ba}) in $(x,y)$ variables. This will be useful in our study of negative eigenvalues of the spectral problem (\ref{M3b}).

Let $\Gamma=\Gamma(x,y)$ be a $\Lambda_0$-periodic function in $D_\eta$.
Consider the function
\begin{equation}\label{J17aa}
F(q,p;\lambda)=\Gamma(q,h)h_p.
\end{equation}
Let $A$ be the operator (\ref{J4aa}). Then
\begin{eqnarray}\label{J17ab}
&&Ah_p=\Big(\frac{\lambda^2h_qh_{qp}}{h_p^2}-\frac{(1+\lambda^2h_q^2)h_{pp}}{h_p^3}\Big)_p
-\lambda^2\Big(\frac{h_{qp}}{h_p}-\frac{h_qh_{pp}}{h_p^2}\Big)_q\nonumber\\
&&=\Big(\frac{1+\lambda^2h_q^2}{2h_p^2}\Big)_{pp}-\lambda^2\Big(\frac{h_q}{h_p}\Big)_{qp}=-\omega'(p).
\end{eqnarray}
Next we have
\begin{equation}\label{J17ac}
F_q=\Gamma h_{qp}+\Gamma_x h_p+\Gamma_y h_qh_p,\;\;\;F_p=\Gamma h_{pp}+\Gamma_y h_p^2.
\end{equation}
Using (\ref{J17ac}) together with (\ref{J17ab}), we get
\begin{eqnarray}\label{J17ad}
&&A(\Gamma h_p)=-\omega'(p)\Gamma+\Big(\frac{\lambda^2 h_q(\Gamma_x+\Gamma_y h_q)}{h_p}-\frac{(1+\lambda^2h_q^2)\Gamma_y}{h_p}\Big)_p
-\lambda^2\Big(\Gamma_x\Big)_q\nonumber\\
&&+\Gamma_y\Big(\frac{\lambda^2h_qh_{qp}}{h_p}-\frac{(1+\lambda^2h_q^2)h_{pp}}{h_p^2}\Big)
-\lambda^2(\Gamma_x+\Gamma_yh_q)\Big(\frac{h_{qp}}{h_p}-\frac{h_qh_{pp}}{h_p^2}\Big)\nonumber\\
&&=-\omega'(p)\Gamma-\lambda^2\Gamma_{xx}-\Gamma_{yy}.
\end{eqnarray}
Thus
\begin{equation}\label{J17b}
AF=-\omega'(p)\Gamma-\lambda^2\Gamma_{xx}-\Gamma_{yy}.
\end{equation}

Since
$$
\psi_{xy}=-\frac{h_{qp}+h_{pp}\psi_X}{h_p^2},\;\;\psi_{yy}=-\frac{h_{pp}\psi_y}{h_p^2},
$$
we have
\begin{eqnarray*}
&&\frac{1+\lambda^2\psi_x\psi_{xy}+\psi_y\Psi_{yy}}{\psi_y}=h_p\Big(1-\lambda^2\psi_x\frac{h_{qp}+h_{pp}\psi_x}{h_p^2}-\psi_y^2\frac{h_{pp}}{h_p^2}\Big)\\
&&=h_p+\lambda^2\frac{h_qh_{qp}}{h_p^2}-\frac{(1+\lambda^2h_q^2)h_{pp}}{h_p^3}.
\end{eqnarray*}
Furthermore
\begin{eqnarray*}
&&-NF=\Big(\frac{\lambda^2h_q(\Gamma h_{qp}+\Gamma_xh_p+\Gamma_yh_qh_p)}{h_p^2}-\frac{(1+\lambda^2h_q^2)(\Gamma h_{pp}+\Gamma_yh_p^2)}{h_p^3}\Big)_{p=1}\\
&&=\Big(\lambda^2\frac{h_qh_{qp}}{h_p^2}-\frac{(1+\lambda^2h_q^2)h_{pp}}{h_p^3}\Big)\Gamma+\lambda^2\frac{h_q}{h_p}\Gamma_x-\frac{1}{h_p}\Gamma_y.
\end{eqnarray*}
Therefore
\begin{equation}\label{J18a}
-NF+F=\sqrt{\psi_x^2+\psi_y^2}\rho \Gamma-\lambda^2\psi_x\Gamma_x-\psi_y\Gamma_y,
\end{equation}
where
\begin{equation}\label{Sept17aa}
\rho=
\rho(x;t)=\frac{(1+\lambda^2\psi_x\psi_{xy}+\psi_y\psi_{yy})}{\psi_y(\psi_x^2+\psi_y^2)^{1/2}}\Big|_{y=\eta(x;t)}.
\end{equation}

\begin{corollary}\label{CorM1} If $\Gamma=\Gamma(x,y)$ satisfies
 the problem
\begin{eqnarray}\label{J17a}
&&(\lambda^2\partial_x^2+\partial_y^2)\Gamma+\omega'(\psi)\Gamma +\theta\frac{1}{\psi_y}\Gamma=0\;\;\mbox{in $D_\eta$},\nonumber\\
&&(\lambda^2\nu_x\Gamma_x+\nu_y\Gamma_y-\rho \Gamma)_{p=1}=0\;\;\mbox{on $B_\eta$},\nonumber\\
&&\Gamma=0\;\;\mbox{for $x=0$},
\end{eqnarray}
where $\nu=(\nu_x,\nu_y)=\nabla\psi/|\nabla\psi|$ is the unite outward normal to $y=\eta(x)$,
then
\begin{eqnarray}\label{J18ba}
&&AF=\theta F,\nonumber\\
&&NF-F=0\;\;\mbox{for $p=1$},\nonumber\\
&&F=0\;\;\mbox{for $p=0$}.
\end{eqnarray}
\end{corollary}

\begin{corollary}\label{CorM2} If $\Gamma=\Gamma(x,y)$ satisfies the problem
\begin{eqnarray}\label{J17aaa}
&&(\lambda^2\partial_x^2+\partial_y^2)\Gamma+\omega'(\psi)\Gamma=0\;\;\mbox{in $D_\eta$},\nonumber\\
&&(\lambda^2\nu_x\Gamma_x+\nu_y\Gamma_y-\rho \Gamma)_{p=1}=\mu\frac{\psi_y}{|\nabla\psi|}\Gamma\;\;\mbox{on $B_\eta$},\nonumber\\
&&\Gamma=0\;\;\mbox{fo $x=0$},
\end{eqnarray}
then
\begin{eqnarray}\label{J18baa}
&&AF=0,\nonumber\\
&&NF-F=\mu F,\;\;\mbox{for $p=1$}\nonumber\\
&&F=0\;\;\mbox{for $p=0$}.
\end{eqnarray}
\end{corollary}

\subsection{Solvability of the Dirichlet problem (\ref{F19aa})}

The next proposition deals with the following Dirichlet problem
\begin{eqnarray}\label{F20b}
&&Aw=f,\nonumber\\
&&w=g,\;\;\mbox{for $p=1$}\nonumber\\
&&w=0\;\;\mbox{for $p=0$}.
\end{eqnarray}

\begin{proposition}\label{PrM1} Let $g\in C^{2,\alpha}_{pe}(\Bbb R)$ and $f\in C^{0,\alpha}_{pe}(Q)$, $\alpha\in (0,\gamma]$. Then the problem (\ref{F20b}) has solution
$w\in C^{2,\alpha}_{pe}(Q)$, which satisfies the estimate
$$
||w||_{C^{2,\alpha}_{pe}(Q)}\leq C(||f||_{C^{0,\alpha}_{pe}(Q)}+||g||_{C^{2.,\alpha}_{pe}(\Bbb R)}).
$$
\end{proposition}
\begin{proof} 
We start from the problem
\begin{eqnarray*}
&&\lambda^2\Gamma_{xx}+\Gamma_{yy}+\omega'(\psi)\Gamma=f\;\;\mbox{in $D_\eta$},\\
&&\Gamma=g\;\;\mbox{for $y=\eta(x)$},\\
&&\Gamma(x,0)=0.
\end{eqnarray*}
This problem is elliptic. Let us show that it has a trivial kernel. Differentiating the first equation in (\ref{Okt6aa})
with respect to $y$ we get
$$
\lambda^2\psi_{xxx}+\psi_{xyy}+\omega'(\psi)\psi_y=0\;\;\mbox{in $D_\eta$}.
$$
Since $\psi_y>0$ in $\overline{D_\eta}$ by maximum principle the kernel of the problem is trivial (see \cite{PrWe}). Moreover the corresponding operator is symmetric and therefore the cokernel is also trivial. These facts and required smoothness of coefficients and standard solvability results for elliptic biundary value problems prove the proposition.
\end{proof}

Next consider the nonlinear problem
\begin{eqnarray}\label{M7a}
&&{\mathcal F}(w;h,t)=f\;\;\mbox{in $Q$},\nonumber\\
&&w=g\;\;\mbox{for $p=1$},\nonumber\\
&&w=0\;\;\mbox{for $p=0$}.
\end{eqnarray}

\begin{proposition}\label{PrA8} There exist positive number $\delta_0$ depending on $M$, $R-\max h(p,1;t)$ such that if
$$
||f||_{C^{0,\gamma}_{pe}(Q)}+||g||_{C^{2,\gamma}_{pe}(\Bbb R)}\leq \delta\leq\delta_0,
$$
then there exists a unique solution $w\in C^{2,\gamma}_{pe}(Q)$ such that
$$
||w||_{C^{2,\gamma}_{pe}(Q)}\leq C\delta .
$$

\end{proposition}
\begin{proof} Since the problem (\ref{M7a}) is a small perturbation of (\ref{F20b}) and the proof is quite standard.
\end{proof}

\subsection{Comparison of two spectral problems}\label{SOkt10b}

Here we will compare the nagative spectrum of the following spectral problems
\begin{eqnarray}\label{F20a}
&&A(t)u=0\;\;\mbox{in $Q$}\nonumber\\
&&N(t) u-u=\mu u\;\;\mbox{for $p=1$}\nonumber\\
&&u=0\;\;\mbox{for $p=0$}.
\end{eqnarray}
and
\begin{eqnarray}\label{F20aa}
&&A(t)u=\theta u\;\;\mbox{in $Q$}\nonumber\\
&&N(t) u-u=0\;\;\mbox{for $p=1$}\nonumber\\
&&u=0\;\;\mbox{for $p=0$}.
\end{eqnarray}

\begin{proposition}\label{Pm23} The spectral problems {\rm (\ref{F20a})} and {\rm (\ref{F20aa})} have the same number of negative eigenvalues (accounting their multiplicities).
\end{proposition}
\begin{proof}
We introduce the bilinear form
\begin{eqnarray*}
&&{\bf a}(u,v)=\int_0^1\int_{-\Lambda_0/2}^{\Lambda_0/2}\Big(\Big(\frac{(1+\lambda^2h_q^2)w_p}{h_p^3}
-\frac{\lambda^2h_qw_q}{h_p^2}\Big)\overline{v_p}+\lambda^2\Big(\frac{w_q}{h_p}-\frac{h_qw_p}{h_p^2}\Big)\overline{v_q}\Big)dqdp\\
&&-\int_{-\Lambda_0/2}^{\Lambda_0/2}w\overline{v}dq.
\end{eqnarray*}
We introduce several spaces. The first one $H_{pe}^1(Q)$ consists of $\Lambda_0$-periodic, even functions in $Q$ from the usual Sobolev space $H^1(Q)$ which are equal to zero for $p=0$. The second space, denoted by $\widehat{H}_{pe}^1(Q)$ is a subspace of $H_{pe}^1(Q)$ satisfying $Au=0$ in $Q$. One more space,  denoted by $\widetilde{H}_{pe}^1(\Omega)$, is the subspace of $H_{pe}^1(Q)$ consisting of functions satisfying $u=0$ for $p=1$. One can verify that
$$
{\bf a}(u,v)=0\;\;\;\mbox{for $u\in \widehat{H}_{pe}^1(Q)$ and $v\in \tilde{H}_{pe}^1(Q)$.}
$$

 Spectra of both spectral problems consists of eigenvalues bounded from below with the only accumulation point at $\infty$.
 Denote by $N_1$ the number (accounting their multiplicities) of negative eigenvalues of the problem (\ref{F20aa}) and similar number for the problem (\ref{F20a}) will be denoted by $N_2$. As is known
 $$
 N_1=\max_{X\subset H_{pe}^1(Q),a(w,w)<0,w\in X\setminus\{0\}}\dim X
 $$
 and
 $$
 N_2=\max_{X\subset \widehat{H}_{pe}^1(Q),a(w,w)<0,w\in X\setminus\{0\}}\dim X.
 $$
 Here $X$ is a finite dimensional subspace of $H_{pe}^1(Q)$ or $\widehat{H}_{pe}^1(Q)$ respectively. Since $\widehat{H}_{pe}^1(Q)$ is a subspace of $H_{pe}^1(Q)$, $N_1\geq N_2$. Furthermore, one can verify that
 \begin{equation}\label{A8a}
 {\bf a}(w,w)>0\;\;\mbox{for $w\in  \widetilde{H}_{pe}^1(Q)\setminus \{O\}$}.
 \end{equation}
 This follows from the positivity of the form
 $$
 \int_{-\Lambda_0/2}^{\Lambda_0/2}\int_0^{\eta(x)}\Big(\lambda^2|u_x|^2+|u_y|^2-\omega'(\psi)|u|^2\Big)dxdy
 $$
 considered on  $\Lambda_0$-periodic, even functions from $H^1(D_\eta)$, vanishing on the boundary of $D_\eta$. This implies (\ref{A8a}), after the change of variables.

Let  $X$ be a finite dimensional subspace of $H_{pe}^1(Q)$ such that $\dim X=N_1$ and ${\bf a}(w,w)<0$ for $w\in X\setminus \{0\}$. We represent $w\in X$ as
$$
w=U+V,\;\;U\in \widehat{H}_{pe}^1(Q)\;\;\mbox{and  $U\in \widetilde{H}_{pe}^1(Q)$}.
$$
We denote the set of such $U$ by $X_0$. Its dimension is equal to $\dim X$. Indeed if $U=0$ for a certain $w\in X$ then $w=U$ and hence $w$ and $w_p$ is equale to $0$ on the boundary $p=1$. This implies that $w=0$. Furthermore, we have
$$
{\bf a}(U+V,U+V)={\bf a}(U,U)+{\bf a}(V,V)<0,
$$
therefore
$$
{\bf a}(U,U)<0\;\;\mbox{for $U\in X_0\setminus \{0\}$\,}.
$$
This proves that $N_1=N_2$.
\end{proof}

Combination of Proposition \ref{Pm23} and Corollaries \ref{CorM1}, \ref{CorM2} gives
\begin{corollary}\label{Pm23a} The number of negative eigenvalues (counting together with their multiplicities) is the same for the spectral problems
{\rm (\ref{F20a})}, {\rm (\ref{F20aa})} and {\rm (\ref{J17a})}, {\rm (\ref{J17aaa})}.
\end{corollary}

\section{On negative eigenvalues of the spectral problem (\ref{M3b})}

Let
$$
\Omega_{\xi}=\{(X,Y)\,:\,X\in (-\Lambda/2,\Lambda/2),\;0<Y<\xi(X)\}
$$
and let $H^1_{pe}(\Omega_\xi)$ be the subspace of $H^1(\Omega_\xi)$ consisting of even in $X$ functions $u$ satisfying $u(-\Lambda/2,Y)=(\Lambda/2,Y)$.
Introduce the form
\begin{equation}\label{J21a}
a_\xi(w,w)=\int_{\Omega_{\xi}} (|\nabla w|^2-\omega'(\Psi)|w|^2)dXdY-\int_{-\Lambda/2}^{\Lambda/2} \rho|w|^2\sqrt{1+(\xi')^2}dX.
\end{equation}
Then the spectral problem (\ref{K2az}) admits the following variational formulation: find nonzero $u\in H^1_{pe}(\Omega_\xi)$ and $\theta\in\Bbb R$ such that
\begin{equation}\label{Okt12a}
a_\xi(u,w)=\theta\int_{\Omega_\xi}\frac{1}{\Psi_{Y}}u\overline{w}dXdY\;\;\mbox{for all $w\in H^1_{pe}(\Omega_\xi)$}.
\end{equation}

 The important observation, which is essential in analysis of subharmonic bifurcations, concerns the spectral problem (\ref{K2az}) and is contained in the following theorem, where $(\Psi_j(X,Y),\xi_j(X),\Lambda_j)$ are defined in Sect. \ref{S24a}.

\begin{theorem}\label{ThMa} Consider the spectral problem {\rm (\ref{K2az})} (or equivalently {\rm (\ref{Okt12a})} with $(\Psi,\xi,\Lambda)=(\Psi_j,\xi_j,\Lambda_j)$. For any positive $\varepsilon$ there are subspaces ${\mathcal X}_j\subset H^1_{pe}(\Omega_{\xi_j})$ such that

(i) $\sup u\in B_\varepsilon ((0,\xi(0)))\cap \Omega_{\xi_j}\;\;\mbox{for $u\in {\mathcal X}_j\setminus \{O\}$}$;

(ii) $a_{\xi_j}(u,u)<0$\;\;\mbox{for $u\in {\mathcal X}_j\setminus\{O\}$};

(iii) $\dim {\mathcal X}_j\to\infty$\;\;\mbox{as $j\to\infty$}.

(iv) If we denote by $N_j$ the number of negative eigenvalues of the problem {\rm (\ref{K2az})} then
$$
N_j\geq \dim {\mathcal X}_j.
$$

\end{theorem}

The proof of this theorem consists of several steps and presented in the remaining part of this section. 
In Sect. \ref{S23b} and \ref{SectA8} we  prove a convergence of $(\Psi (X,Y;t_j),\xi(X;t_j),\Lambda_j)$ to a Stokes extreme wave. The negative spectrum of the Frechet derivative corresponding to the extreme waves is studied in \cite{KL2} and using this result we complete the proof of Theorem \ref{ThMa} in Sect.\ref{SA21a}. 

\subsection{Some estimates for solutions to (\ref{Okt6aa})}\label{S23b}

Let
$$
\omega_1=\max_{0\leq p\leq 1}(|\omega(p)|+|\omega'(p)|).
$$
The following proposition is a local version of Proposition 3 from \cite{KLN17}.

\begin{proposition}\label{Mp1}  Let  $\delta >0$ be given as well as a disc $B$ of radius $\rho>0$ and let $C_B=\sup_{(X,\xi(X))\in B} |\xi'(X)|$. Let also  $\inf_B\Psi_Y\geq\delta$. Then there exists constants $\widehat{\gamma}\in (0,1]$ and $C>0$, depending only on $R$, $\delta$, $\omega_1$, $\rho$ and $C_B$ such that any solution $(\Psi,\xi)\in C^{1}(\overline{{\mathcal D}_\xi}\cap\overline{B})\times C^{0,1}(\Bbb R)$  of {\rm (\ref{K2a})} satisfies
\begin{equation}\label{M25a}
||\Psi||_{C^{2,\widehat{\gamma}}({\mathcal D}_\xi\cap\frac{1}{2}B)}\leq C,\;\;||\xi||_{C^{2,\widehat{\gamma}}(I)}\leq C,
\end{equation}
where $\frac{1}{2}B$ is the disc with the same center and radius $\frac{1}{2}\rho$ and $I$ is the projection of $\frac{1}{2}B$ on $x$-axis.

\end{proposition}

We recall  that the limit functions $(\xi_*,\Psi_*,\Lambda_*)$ satisfy
\begin{eqnarray}\label{M19a}
 &&\Delta\Psi_*+\omega(\Psi_*)=0\;\;\mbox{in ${\mathcal D}_{\xi_*}$},\nonumber\\
 &&\frac{1}{2}|\nabla\Psi_*|^2+\xi_*=R\;\;\mbox{on ${\mathcal B}_{\xi_*}$},\nonumber\\
 &&\Psi_*(X,0)=0,\;\;\Psi_*(X,\xi_*(X))=1
 \end{eqnarray}
 in a weak sense. Due to the assumptions (see Sect.\ref{S24a})
\begin{equation}\label{M26a}
\lim_{X\to 0+}\frac{\xi_*(X)-R}{X}=-\frac{1}{\sqrt{3}}
\end{equation}
 and by Proposition \ref{Mp2} $\Psi_{*Y}>0$ in $\overline{{\mathcal D}_{\xi_*}}\setminus \{(0,R)\}$. Using Proposition \ref{Mp3}, we get
 \begin{equation}\label{M26aa}
 \xi_*\in C^{2,\widehat{\gamma}},\;\;\Psi_*\in C^{2,\widehat{\gamma}}\;\;\mbox{outside stagnation points}.
 \end{equation}
 The following theorem is proved  in \cite{KL2}.
\begin{theorem}\label{TA30}
Assume that $\Psi_{*Y}>0$ for $(X,Y)\in \overline{{\mathcal D}_{\xi_*}}$ and $Y\neq R$ and $\xi_*$ satisfies (\ref{M26a}). Then
\begin{equation}\label{A30aa}
\xi_{*X}(X)=-\frac{1}{\sqrt{3}}+\frac{2^{\frac{3}{2}}}{3^{\frac{5}{4}}}\omega(1)\sqrt{X}+a_1\omega^2(1)X+f(X)
\end{equation}
for some explicit $a_1>0$, where $f=O(X^{\frac{3}{2}(\tau_1-1)})$, $f_X=O(X^{\frac{3}{2}(\tau_1-1)-1})$ as $X\rightarrow 0+$ and $\tau_1\approx 1.8$ is the smallest root of $\tau_1=-\frac{1}{\sqrt{3}}\cot (\frac{\pi}{2}\tau_1)$.
\end{theorem}

Using (\ref{Ma8a}), we derive from the Bernoulli equation
\begin{equation}\label{M21bc}
\Psi_{jY}(X,\xi_j(X))^2\geq \frac{2}{1+{\mathcal B}_j^2}(R- \xi_j(X)),\;\;{\mathcal B}_j=\frac{\Lambda_0}{\Lambda_j}{\mathcal B},
\end{equation}
where $(\Psi_j,\xi_j,\Lambda_j)$ was introduced in Sect.\ref{S24a}. We put ${\mathcal B}_*=\max_j{\mathcal B}_j$. Due to the convergence $\lambda_j\to\lambda_*$,  ${\mathcal B}_*<\infty$.

\begin{proposition}\label{Mp2} Let $\delta>0$ be a small number and $X_*=\delta$. Then
there exist positive constants $c_k$, $k=1,\ldots,5$, and the index $j(\delta)$ such that $|X-X_*|\leq c_1\delta$ and $j\geq j(\delta)$  implies
\begin{equation}\label{Ma8baa}
|R-\xi_j(X)|\geq c_2\delta
\end{equation}
and
\begin{equation}\label{Ma9a}
\Psi_{jY}(X,Y)\geq c_4\delta\;\;\mbox{for   $|Y-\xi_j(X)|\leq c_5\delta$}.
\end{equation}
Here the constants $c_k$, $k=1,\ldots,5$, are independent of $\delta$ and $j$.
\end{proposition}
\begin{proof}
Since $\xi_j\to\xi_*$ in $C^{0,\alpha}_{pe}(\Bbb R)$ as $j\to\infty$, for every $\varepsilon>0$ there exists $j(\varepsilon)$ such that
\begin{equation}\label{Ma8b}
|\xi_j(X)-\xi_*(X')|+\Big |\frac{\xi_j(X)-\xi_j(X')}{|X-X'|^\alpha}-\frac{\xi_*(X)-\xi_*(X')}{|X-X'|^\alpha}\Big|\leq\varepsilon
\end{equation}
for $j\geq j(\varepsilon)$ and $0\leq X',X\leq 2\delta$.
We choose $\alpha=\frac{3}{2}(\tau_1-1)-1$. Applying the asymptotic formula (\ref{A30aa}), we get
\begin{equation}\label{Ma9b}
|\xi_*(X)-\xi_*(X')+c_*(X-X')|\leq c_6\delta^{\alpha}|X-X'|,\;\;c_*=\frac{1}{\sqrt{3}}.
\end{equation}
Now using (\ref{Ma8b}) and (\ref{Ma9b}), we obtain
\begin{equation}\label{Ma8ba}
R-\xi_j(X)=\xi_*(0)-\xi_*(X_*)+\xi_*(X)-\xi_j(X)\geq c_*\delta-\varepsilon-c_6\delta^{1+\alpha}.
\end{equation}
By choosing $\varepsilon=c_*\delta/2$ and $c_1$ and $c_6$ sufficiently small we arrive at (\ref{Ma8baa}).

Let us turn to proving (\ref{Ma9a}). We shall omit the index $j$ in the proof of proposition.
Introduce  function $u=e^{-sY}\Psi_Y$
Then
$$
\Delta u+2s\partial_{Y}u=-(s^2+\omega'(\Psi))u\;\;\mbox{and}\;\;\Psi_Y>0\;\;\mbox{in ${\mathcal D}_\xi$}.
$$
Choosing $s$ to be non-negative and $s^2\geq -\min_{p\in [0,1]}\omega'(p)$, we get $\Delta u+2s\partial_{Y}u\leq 0$. We will compare the function $u$ with the barrier function
$$
U(X,Y)=\cos (s_*(X-X_*))\sinh (s_*(Y-a)).
$$
Then
$$
\Delta U+2s\partial_{Y}U=2s\cos (s_*(X-X_*))\cosh (s_*(Y-a).
$$
Furthermore, $U=0$ for $X-X_*=\pm X_*$, where $s_*X_*=\pi/2$, and $U=0$ for $Y=a$. Consider the function
$$
V=u-\alpha U\;\;\mbox{with $\alpha>0$}.
$$
Then
$$
\Delta V+2s\partial_{Y}V\leq 0\;\;\mbox{for $|X-X_*|\leq X_*$ and for $a\leq Y\leq \xi(X)$}
$$
Furthermore $V>0$ for $X=\pm X_*$ and for $Y=a$. From (\ref{M21bc}) and (\ref{Ma8baa}) it follows that
$$
\Psi_Y(X,\xi(X))\geq c_7\delta\;\;\mbox{for $|X-X_*|\leq c_1\delta$}.
$$
Choosing $\alpha$ to satisfy
$$
e^{-s_*\xi(X)}c_7\delta\geq \alpha \cosh(s_*(\xi(X)-a))\;\;\mbox{for $|X-X_*|\leq c_1\delta$}
$$
we arrive at (\ref{Ma9a}).


\end{proof}
Combination of Propositions \ref{Mp1} and \ref{Mp2} gives the following

\begin{proposition}\label{Mp3} Let the assumptions of {\rm Proposition \ref{Mp2}} be satisfied and $\delta=|X_*|$. Then the inequality {\rm (\ref{M25a})} is valid with
$B=B_{c\delta} (X_*,\xi_j(X_*))$, where $c$ is independent of $j$ and $\delta$.
\end{proposition}

\begin{corollary} Theorem \ref{TA30} together with the theorem \ref{Mp3} inplies that $\xi_*\in C^{2,\gamma'}(I_c)$ for certain $\gamma'\in (0,1)$, where $I_c$ is any closed interval which does not contain stagnation points inside.
\end{corollary}

\subsection{On convergence of the branch of Stokes waves (\ref{J4ac})}\label{SectA8}

Due to Proposition \ref{Mp3} and the assumption (\ref{M26a}), we conclude that
\begin{equation}\label{M29aa}
||\Psi_j||_{C^{2,\widehat{\gamma}}({\mathcal D}_{\xi_j}\cap\frac{1}{2}B)}\leq C,\;\;||\xi_j||_{C^{2,\widehat{\gamma}}(I)}\leq C,\;\;j=1,\ldots,
\end{equation}
for every disc $B=B_{c\delta} (X,\xi_j(X))$, where $c$ is independent of $j$ and $\delta$  and $C$ depending only on $\delta$, $R$ and $\omega_1$.

For $\delta>0$ we introduce
$$
\Omega_{\xi_j}^\delta=\{(X,Y)\,:\,X\in [-\Lambda_j/2,\Lambda_j/2],\,Y\in [0,\min(\xi_j(X),R-c\delta)]\}
$$
and
\begin{equation}\label{Ma10a}
I^\delta_j=\{X\,:\,X\in [-\Lambda_j/2,-\delta]\cup [\delta,\Lambda_j/2]\}.
\end{equation}

In the next proposition we present an important additional information on convergence of the sequence $(\Psi_j,\xi_j,\Lambda_j)$.
\begin{proposition}\label{PA30} Let  $\delta$ be a positive number.
Then
\begin{equation}\label{Ma7aa}
||\Psi_j||_{C^{2,\widehat{\gamma}}(\Omega_{\xi_j}^\delta)}\leq C\;\;\mbox{and}\;\;||\xi_j||_{C^{2,\widehat{\gamma}}(I^\delta_j)}\leq C,
\end{equation}
where the constant $C$ depends on $\delta$, $R$, $\omega_1$ and ${\mathcal B}_*$.
\end{proposition}
\begin{proof} Let $B$ be a disc with the center located at a distance $\geq \delta$ from the stagnation point $(0,R)$ and with a radius $c\delta$ with a certain constant $c$ independent of $j$ and $\delta$. Then using the local estimates 
(\ref{M25a}) and verifying conditions of Proposition \ref{Mp1} by using Proposition \ref{Mp2} we obtain estimates (\ref{M25a}) for functions $\Psi_j$ and $\xi_j$ with constant $C$ depending on $\delta$, $R$, $\omega_1$ and ${\mathcal B}_*$. These local estimates leads to the  estimate
(\ref{Ma7aa}).

\end{proof}

By choosing a subsequence of the indexes $\{j\}$ (and numerating it by the same indexes) we can assume that for every $\delta>0$
\begin{equation}\label{A2aa}
\xi_j\to\xi_*\;\;\mbox{in $C^{2,\alpha}(\overline{I^\delta_*})$}
\end{equation}
and
\begin{equation}\label{A2ab}
\rho_j\to\rho_*\;\;\mbox{in $C^{0,\alpha}(\overline{I^\delta_*})$}\;\;\mbox{as $j\to\infty$},
\end{equation}
where $I^\delta_*$ is defined by (\ref{Ma10a}) where $\Lambda_j$ is replaced by $\Lambda_*$-the period of $\xi_*$.
Furthermore $\alpha\in (0,\widehat{\gamma})$, $\rho_j(X)=\rho(X;t_j)$  and $\rho=\rho(x;t)$ is defined by (\ref{Sept17aa}) with $\lambda=1$, i.e.
\begin{equation}\label{A27b}
\rho=
\rho(X;t)=\frac{(1+\Psi_X\Psi_{XY}+\Psi_Y\Psi_{YY})}{\Psi_Y(\Psi_X^2+\Psi_Y^2)^{1/2}}\Big|_{Y=\xi(X;t)}.
\end{equation}

\begin{remark}
As it was shown in \cite{KL2} the following formula is valid for $\rho_*$:
\begin{equation}\label{A28b}
\rho_*(X)=r^{-1}\frac{1}{\sqrt{3}}+O(r^{-1+\varepsilon}),\;\;\mbox{as $r\to 0$},
\end{equation}
where $r^2=(R-X)^2+Y^2$.
\end{remark}

\subsection{Proof of Theorem \ref{ThMa} on negative eigenvalues}\label{SA21a}

\begin{proof}   Let $a_j(u,w)=a_{\xi_j}(u,w)$.

Let also $(r,\theta)$ be the polar coordinates with center at $(0,R)$, i.e. $r=\sqrt{(R-X)^2+Y^2}$ and $\theta $ is the angle measured from the ray $(X,R)$, $X>0$.
We introduce $\widehat{\xi}(X)=R-|R-X|/\sqrt{3}$. We will compare the negative eigenvalues of the problem (\ref{J17az}) with the negative spectrum of the problem
\begin{eqnarray}\label{A28a}
&&-\Delta u=-\tau^2 u\;\;\mbox{in ${\mathcal D}_{\widehat{\xi}}$}\nonumber\\
&& \partial_\theta u-\frac{\sqrt{3}}{2} u=0\;\;\mbox{for $Y=\widehat{\xi}(X)$}
\end{eqnarray}
considered for even functions and supplied with a condition
$$
u=\sin\Big(\kappa\log(\frac{1}{2}r)+\gamma\Big)\cosh(\kappa\theta)+O(r^\epsilon)\;\;\mbox{near zero},
$$
where $\gamma\in (0,\pi]$ is a fixed number and $\epsilon$ is a small positive number. Here $\kappa$ is the positive root of
\begin{equation}\label{K49}
\kappa\tanh\Big(\frac{\kappa\pi}{3}\Big)=\frac{\sqrt{3}}{2}
\end{equation}
Let $K_{i\kappa}(z)$ be Bessel's function of imaginary order, see \cite{Dan}. According to \cite{KL2} the spectral problem (\ref{A28a})
has infinitely many negative eigenvalues $\lambda=-\tau_k^2$,
$$
\tau_k=2e^{(\gamma_\kappa+\gamma)/\kappa}e^{k\pi/\kappa},\;\;j\in{\Bbb Z}.
$$
Corresponding eigenfunction is
$$
u_k(X,Y)=K_{i\kappa}(\tau_kr)\cos(\kappa\theta).
$$

 We recall that
 $$
K_{i\kappa}(z)=\Big(\frac{\pi}{2z}\Big)^{1/2}e^{-z}\Big(1+O\Big(\frac{1}{z}\Big)\Big),\;\;
$$
for large $z$ and
\begin{equation}\label{F23a}
K_{i\kappa}(z)=-\Big(\frac{\pi}{\kappa\sinh(\pi\kappa)}\Big)^{1/2}
\sin\big(\kappa\ln\big(\frac{1}{2}z\big)-\gamma_\kappa\big)+O(z^2),
\end{equation}
for small $z$.
Here  $\gamma_\kappa$ is a real constant defined by
\begin{equation}\label{J25a}
\Gamma(1+i\kappa)=\Big(\frac{\pi\kappa}{\sinh(\pi\kappa)}\Big)^{1/2}e^{i\gamma_\kappa}.
\end{equation}
We need  the following asymptotic formula for small roots of $K_{i\kappa}(z)=0$:
\begin{equation}\label{A2a}
z_j=2e^{-(j\pi-\gamma_\kappa)/\kappa}(1+O(e^{-2(j\pi-\gamma_\kappa)/\kappa}))\;\;\mbox{as $j\to\infty$}.
\end{equation}
We note also that the function $K_{i\kappa}(z)$ is bounded and
$$
\int_0^\infty |K_{i\kappa}(\tau r)|^2rdr=c\tau^{-2}.
$$

For small $\varepsilon>0$ we define the function
$$
u_{k,\varepsilon}(x,y)=K_{i\kappa}(\tau_kr)\cos(\kappa\theta) \;\;\mbox{for $r>z_n/\tau_k$}.
$$
where $z_n$ is the smallest root of $K_{i\kappa}(z)=0$ such that $z_n\tau_k^{-1}>\varepsilon$. If $r<z_n/\tau_j$ then $u_{k,\varepsilon}(X,Y)=0$.
We choose small parameters
 $$
 \epsilon_1<\epsilon<\sigma<\delta.
 $$

According to (\ref{A2aa}) and (\ref{A2ab}) for every $\varepsilon_1>0$ we have
\begin{equation}\label{A2b}
|\xi_j(X)-\xi_*(X)|\leq \varepsilon_1\;\;\;\mbox{for $X\in (\varepsilon,\delta)$}
\end{equation}
and
\begin{equation}\label{A2ba}
|\rho_j(X)-\rho_*(X)|\leq \varepsilon_1\;\;\;\mbox{for $X\in (\varepsilon,\delta)$}
\end{equation}
for $j\geq j(\varepsilon_1)$, where $j(\varepsilon_1)$ is sufficiently large.

We introduce also a smooth cut-off function $\zeta=\zeta(r)$, $\zeta(r)=1$ for $r\leq 1$ and $\zeta(r)=0$ for $r\geq 2$ and put
$$
\zeta_\delta(r)=\zeta(r/\delta).
$$

To verifies that there are many negative eigenvalues of the problem (\ref{J17a}), we will use the finite dimensional space ${\mathcal X}$ of test functions
\begin{equation}\label{M31aa}
w(X,Y)=\zeta_\delta(r)\sum a_ku_{k,\varepsilon}(X,Y).
\end{equation}
We assume in what follows that $k$ is chosen to satisfy
\begin{equation}\label{M31a}
\sigma^{-1}\leq\tau_k\leq \sigma\varepsilon^{-1},
\end{equation}
where $\sigma$ is a small number. 
We use the norm
$$
||w||_{\mathcal X}:=|{\bf a}|=\Big(\sum |a_k|^2\Big)^{1/2}
$$
in the space ${\mathcal X}$. Here ${\bf a}$ denotes the vector of coefficient $a_k$. 
$$
$$
One can verify that $N:=\dim {\mathcal X}\approx \log \varepsilon^{-1}$.

 We represent the form (\ref{J21a}) as
$$
a_j(w,w)=a(w,w)+b_j(w,w)
$$
where
\begin{equation}\label{J21ad}
a(w,w)=\int_{\Omega_{\xi_*}} |\nabla w|^2dXdY-\int_{-\Lambda_0/2}^{\Lambda_0/2} \rho_*|w|^2\sqrt{1+(\eta_*')^2}dX
\end{equation}
and
\begin{eqnarray*}
&&b_j(w,w)=\int_{\Omega_{\eta_j}\setminus \Omega_{\xi_*}}|\nabla w|^2dXdY-\int_{D_{\xi_j}}\omega'(\Psi_j)|w|^2dXdY\\
&&+\int_{-\Lambda/2}^{\Lambda/2}(\rho_*|w(X,\xi_*(X)|^2\sqrt{1+\xi_*'(x)^2}-\rho_j|w(X,\xi_j(X)|^2\sqrt{1+\eta_j'(x)^2})dX
\end{eqnarray*}

Using (\ref{A2b}) and (\ref{A2ba}), we get
$$
|b_j(w,w)|\leq cN|{\bf a}|^2\Big(\sigma^2+\frac{\varepsilon_1}{\varepsilon}+\varepsilon_1\log\frac{\delta}{\varepsilon}\Big),
$$
where $|{\bf a}|$ is the euclidian norm of the vector ${\bf a}$ whose components are the coefficients in (\ref{M31aa}).
Next, we represent the form $a$ as
$$
a(w,w)=\widehat{a}(w,w)+\widehat{b}(w,w),
$$
where
\begin{equation}\label{M31b}
\widehat{a}(w,w)=\int_{\Omega_{\widehat{\xi}}} |\nabla w|^2dXdY
-\int_{-\Lambda_0/2}^{\Lambda_0/2}\widehat{\rho}|w|^2\sqrt{1+(\widehat{\xi}')^2}dX
\end{equation}
and
\begin{eqnarray}\label{M31ba}
&&\widehat{b}(w,w)=\int_{\Omega_{\xi_*}\setminus \Omega_{\widehat{\xi}}}|\nabla w|^2dXdY\nonumber\\
&&+\int_{-\Lambda_0/2}^{\Lambda_0/2}(\widehat{\rho}|w(X,\widehat{\xi}(X)|^2\sqrt{1+\widehat{\xi}'(X)^2}-\rho_*|w(X,\xi_*(X)|^2\sqrt{1+\xi_*'(X)^2})dX
\end{eqnarray}
where, according to \cite{KL2},
$$
\widehat{\rho}=\Big(\frac{\sqrt{3}}{2}+O(X^{1/2})\Big)r^{-1}.
$$
This implies, in particular (see also \cite{KL2})
$$
\rho_*-\widehat{\rho}=O(X^{-1/2}).
$$
Therefore
$$
|\widehat{b}(w,w)|\leq C|a|^2N\sigma^{1/2}.
$$
Next, we introduce
\begin{equation}\label{M31aaa}
W_\varepsilon (X,Y)=\sum a_ku_{k\varepsilon}(X,Y)
\end{equation}
and put
\begin{equation*}
\widetilde{a}(W_\varepsilon,W_\varepsilon )=\int_{\Omega_{\widehat{\xi}}} |\nabla W_\varepsilon|^2dXdY
-\int_{-\Lambda_0/2}^{\Lambda_0/2}\widehat{\rho}|W_\varepsilon|^2\sqrt{1+(\widehat{\xi}')^2}dX.
\end{equation*}
Then
\begin{equation*}
\widehat{a}(w,w)=\widetilde{a}(W_\varepsilon,W_\varepsilon)+\widetilde{b}(w,W_\varepsilon),
\end{equation*}
where
$$
\widetilde{b}(w,W_\varepsilon)=\widehat{a}(w,w)-\widetilde{a}(W_\varepsilon,W_\varepsilon).
$$
Then
$$
|\widetilde{b}(w,W_\varepsilon)|\leq C|a|^2N\Big(e^{-\delta/\sigma}+\sigma\Big)
$$
Furthermore

$$
a(W_\varepsilon,W_\varepsilon)=-c_0|a|^2(1+O(N\sigma^2)),\;\;c_0=\int_0^\infty K_{i\kappa}(z)^2zdz.
$$
Choosing $\varepsilon_1$ sufficiently small and
$$
\sigma=\varepsilon^{1/4},\;\;\mbox{and}\;\;\delta=\epsilon^{1/8}\;\;\mbox{and}\;\;N=\log\varepsilon^{-1/2},
$$
we arrive at
\begin{equation}\label{A29b}
a_j(w,w)\geq -c_0|a|^2/2
\end{equation}
if $\varepsilon$ is sufficiently small. This proves the required result.
\end{proof}

\begin{remark}\label{R28} We note that the test function $w$, used for estimating of negative spectrum, has support in a $2\sigma$ neighborhood of the crest $(0,R)$ of the extreme Stokes wave and so it is zero below the trough.

\end{remark}

\begin{corollary}\label{SOkt11a} Due to Proposition \ref{Pm23} and Corollaries \ref{CorM1}, \ref{CorM2} all spectral problems (\ref{J17a}), (\ref{J17aaa}), (\ref{F20a}) and (\ref{F20aa}) have the same number of negative eigenvalues. Moreover, one can see directly that they have the same multiplicity of zero eigenvalue.

\end{corollary}

Our study of bifurcations will be based on the  equation (\ref{F19b}), where ${\mathcal S}$ is given by (\ref{F19ab}).
Using Proposition \ref{PrA8}, we conclude that
\begin{equation}\label{M5a}
{\mathcal S}(t)\,:\, C^{2,\gamma}_{pe}(\Bbb R)\cap {\mathcal U}\rightarrow C^{1,\gamma}_{pe}(\Bbb R),
\end{equation}
where ${\mathcal U}$ is a neighborhood of origin in $C^{2,\gamma}_{pe}(\Bbb R)$.
This branch consists of Fredholm potential operators of index zero analytically depending on the parameter $t\in [0,\infty)$ (we write here and in what follows "analytic" having in mind it has an real analytic re-parametrization near every point $t\geq 0$).  The operator (\ref{M5a}) has potential
\begin{equation}\label{A5a}
{\mathcal G}(q;h+w,\lambda)=f(h+w(g);\lambda)-f(h;\lambda),
\end{equation}
where $w=w(g)$ is the solution of the problem (\ref{F19aa}).
Consider the equation (\ref{F19b}):
\begin{equation}\label{A5aa}
{\mathcal S}(g;t)=0.
\end{equation}

Important role in analysis of bifurcations is played by the Frechet derivative, which is given by the operator (see (\ref{M3a})):
$$
Sg=(Nw-w)|_{p=1},
$$
where $w$ solves the problem (\ref{F19bb}) and $N$ is defined by (\ref{J4ab}). Clearly
\begin{equation}\label{M5aa}
S(t)\,:\,C_{pe}^{2,\gamma}(\Bbb R)\rightarrow  C_{pe}^{1,\gamma}(\Bbb R).
\end{equation}


As is known every point $t$ where
 the operator $S(t)$ has non-trivial kernel is isolated.
 The following result  directly follows from Theorem \ref{ThMa} and Corollary \ref{SOkt11a}

\begin{corollary}\label{TA5z} There are infinitively many  points $\{T_j\}_{j=1}^\infty$ such that $T_j\to\infty$, as $j\to\infty$, the
kernel of $S(T_j)$ is nontrivial and the crossing number of the family $S(t)$ through 0 at $T_j$ is non-zero.

\end{corollary}

\section{Subharmonic bifurcations}

Unfortunately Corollary \ref{TA5z} does not guarantees that the bifurcation points $T_j$ generates subharmonic bifurcations. They can generate bifurcations  which are described by Stokes waves of the same period. In this section we formulate an equation for finding subharmonic bifurcation
and present some important properties for the Frechet derivative of operators involved in the equation.

\subsection{Equation for subharmonic bifurcations}\label{SA22a}

For  $M=2M_1+1$, $M_1=1,2,\ldots$, and $\alpha\in (0,1]$ let us introduce  the subspaces $C_{Me}^{k,\alpha}(\overline{Q})$ and  $C_{Me}^{k,\alpha}(\Bbb R)$ of $C^{k,\alpha}(\overline{Q})$ and  $C^{k,\alpha}(\Bbb R)$ respectively consisting of even functions of period $M\Lambda_0$. A similar space for the domain ${\mathcal D}_\xi$ we denote by $C_{Me}^{2,\alpha}(\overline{{\mathcal D}_\xi})$.

Equation (\ref{F19b}) can be considered also on functions of period $M\Lambda_0$. Having this in mind we write as before $h+w$ instead of $h$ but now we assume that $w\in C^{2,\gamma}_{Me}(\overline{Q})$. Now the operator ${\mathcal F}(w;h,t)$, ${\mathcal G}(w;h,t)$, $A(t)$ and $N(t)$ from Sect.\ref{SA23a} are considered on functions from $C^{2,\gamma}_{Me}(\overline{Q})$. To indicate this difference we will use the notations
${\mathcal F}_M(w;h,t)$, ${\mathcal G}_M(w;h,t)$, $A_M(t)$ and $N_M(t)$ for corresponding operators. Moreover we can define analogs of operators ${\mathcal S}$ and $S$ and denote them by ${\mathcal S}_M$ and $S_M$ correspondently. New operators are acting on functions from $C^{1,\gamma}_{Me}(\Bbb R)$. The equation for subharmonic bifurcations is
\begin{equation}\label{A23a}
{\mathcal S}_M(g;t)=0,
\end{equation}
where
\begin{equation}\label{M5b}
{\mathcal S}_M(g;t)\,:\, C_{Me}^{2,\gamma}(\Bbb R)\rightarrow  C_{Me}^{1,\gamma}(\Bbb R).
\end{equation}
The operator ${\mathcal S}_M$ is also potential and for the definition of the potential the integration in (\ref{M2a}) must be taken over the interval $(-M\Lambda_0/2,M\Lambda_0/2)$.
The corresponding eigenvalue problem for the Frechet derivative $S_M(t)$ is
\begin{equation}\label{A23aa}
S_M(t)g=-\mu g,
\end{equation}
where $S_M$ is defined by the same formulas as $S$ but now the functions $g$ and $w$ are $M\Lambda_0$-periodic. Clearly,
\begin{equation}\label{M5ba}
S_M(t)\,:\,C_{Me}^{2,\gamma}(\Bbb R)\rightarrow  C_{Me}^{1,\gamma}(\Bbb R).
\end{equation}



\subsection{Estimates of negative eigenvalues of $S_M(t)$}

Let
$$
\Omega_{M,t}=\{(X,Y)\,:\, X\in (-M\Lambda/2,M\Lambda/2),\;0<Y<\xi(X)\}
$$
We  introduce the bilinear form
\begin{eqnarray}\label{A29a}
&&{\bf a}_{M,t}(w,v)=\int_{\Omega_{M,t}}(\nabla w\nabla\overline{v}-\omega'(\Psi)w\overline{v})dYdX\nonumber\\
&&-\int_{-M\Lambda/2}^{M\Lambda/2}\rho(X;t)w(X,\xi(X;t))\overline{v}(X,\xi(X;t))\sqrt{1+\xi'^2}dX,
\end{eqnarray}
where the functions $\xi=\xi(X;t)$, $\Psi=\Psi(X,Y;t)$, $\Lambda=\Lambda(t)$ and $\rho=\rho(X;t)$ depends on $t$.
The form (\ref{A29a}) defined on functions from $H^1_{Me}(\Omega_{M,t})$, where
the index $Me$ means that only even functions are taken from the Sobolev space $H^1(\Omega_{M,t})$  which have the same values for $X=\pm M\Lambda/2$. The form ${\bf a}_{M,t}$ corresponds to the self-adjoint operator in the spectral problem
\begin{eqnarray*}
&&(\partial_X^2+\partial_Y^2)w+\omega'(\Psi)w +\theta\frac{1}{\Psi_Y}w=0\;\;\mbox{in ${\mathcal D}_{\xi}$},\\
&&\nu_xw_x+\nu_yw_y-\rho w=0\;\;\mbox{on ${\mathcal B}_{\xi}$},\\
&&w=0\;\;\mbox{for $X=0$},
\end{eqnarray*}
defined on $M\Lambda$-periodic, even functions.
We decompose the space $H^1_{Me}(\Omega_{M,t})$ in the orthogonal sum
\begin{equation}\label{A21a}
H^1_{Me}(\Omega_{M,t})=H^1_e(\Omega_{M,t})+{\mathcal X}_M,
\end{equation}
where $H^1_e(\Omega_{M,t})$ consists of restrictions of functions from $H^1_{pe}({\mathcal D}_\xi)$ onto $\Omega_{M,t}$.  
The restriction of the form ${\bf a}_{M,t}$ to the  subspace $H^1_e(\Omega_{M,\eta})$ corresponds to the spectral problem (\ref{J17a}) define on even $\Lambda$ periodic functions and so this is invariant subspace for the operator defined on $H^1_{Me}(\Omega_{M,t})$.
Therefore the  space ${\mathcal X}$ is orthogonal to the subspace $H^1_e(\Omega_{M,t})$ with respect to the inner product
$$
{\bf b}_{M,t}(w,v)=\int_{\Omega_{M,t}}w\overline{v}\frac{1}{\Psi_Y}dXdY
$$
and with respect to the form ${\bf a}_{M,t}$. So the space ${\mathcal X}_M$ is invariant for the operator $S_M(t)$ and
we denote the restriction operator $S_M(t)$ on ${\mathcal X}_M$ by $\widehat{S}_M(t)$.

Similar to Proposition \ref{Pm23} and Corollary \ref{Pm23a} one can show that

\begin{proposition}\label{Pm26a} The number of negative eigenvalues of the operator $S_M(t)$ and of the operator corresponding to the form ${\bf a}_{M,t}$ defined on $H^1_{Me}(\Omega_{M,t})$ is the same (accounting their multiplicities).
\end{proposition}

\begin{proposition}\label{PrA29b} Let $\widehat{S}_M(t)$ be the operator defined above on ${\mathcal X}_M$. Then

(i) the number of negative eigenvalues of the operator $\widehat{S}_M(t)$ is estimated from above by $C_*(t)M$, where $C_*(t)$ depends on $t$ but it is independent of $M$;

(ii) for every $N>0$ there exist the index $j(N)$ such that the number of negative eigenvalue of the operator $\widehat{S}_M(t_{j(N)})$ is estimated from below by $MN$. Moreover $j(N)\to\infty$ as $N\to\infty$.

\end{proposition}

\begin{proof} (i) Consider the problem (for $M=1$)
\begin{eqnarray}\label{A25a}
&&(\partial_X^2+\partial_Y^2)w+\omega'(\Psi)w+\theta\frac{1}{\Psi_Y}w=0\;\;\mbox{in $\Omega_t$},\nonumber\\
&&\nu_Xw_x+\nu_Yw_y-\rho w=0\;\;\mbox{for $Y=\xi(X)$,\,$X\in(-\Lambda/2,\Lambda/2)$},\nonumber\\
&&w=0\;\;\mbox{for $X=0$},\nonumber\\
&&w_X=0\;\;\mbox{for $X\pm\Lambda/2$ and $0<Y<\xi(\pm\Lambda/2)$},
\end{eqnarray}
where $(\Psi,\xi,\Lambda)$ depends on $t$.
Denote by $C_*=C_*(t)$ the number od negative eigenvalue of this operator. Then the number of negative eigenvalues of the operator $\widehat{S}_M(t)$ is estimated from above by $C_*M$ which proves (i).

(ii) Represent $M$ as $2M_1+1$ and introduce domains
$$
\Omega_{j,k}=\{(X,Y)\;:\;X\in (\Lambda_j/2+(k-1)\Lambda_j,\Lambda_j/2+k\Lambda_j),\;0<Y<\xi_j(X)\}\;\;k=1,\ldots,M_1,
$$
and
$$
\Omega_{j,-k}=\{(X,Y)\;:\;X\in (-\Lambda_j/2+(-k+1)\Lambda_j,-\Lambda_j/2-k\Lambda_j),\;0<Y<\xi_j(X)\}\;\;k=1,\ldots,M_1.
$$
Let $w$ be the same function as in the proof of Theorem \ref{ThMa}.  Consider the functions
$$
w_k(X,Y)=w(X-(2k-1)\Lambda_j,Y)+w(x+(2k-1)\Lambda_j,Y)-(w(x-2k\Lambda_j,Y)-w(x+2k\Lambda_j,Y),
$$
where $k=1,3,5,\ldots,M_2,$ and $2M_2\leq M_1$. Then $w_k\in {\mathcal X}_M$ (here we can refer to Remark \ref{R28} concerning smoothness of this function) and $w_k$ are pairwise orthogonal. Due to the calculation performed in the proof of Theorem \ref{ThMa}, we get
$$
a_{M,t_j}(\sum_k\alpha_kw_k,\sum_k\alpha_kw_k)<0
$$
provided not all coefficients $\alpha_k$ are zero. This implies that the number of negative eigenvalues is estimated by
$$
N_jL_2.
$$
where $N_j$ the number of coefficient in the definition of the function $w$, which tends to infinity when $j\to \infty$. This implies (ii).

\end{proof}


\begin{proposition}\label{RA19a} For each $N>0$ there exists $T_\dagger$ such that if $t_j>T_\dagger$ then the operators $\widehat{S}_M(t_j)$ have at least $N$ negative eigenvalues for all odd $M$.

\end{proposition}
\begin{proof} If $M=1$ the result is proved in Theorem \ref{ThMa}.  Let $M=2M_1+1$, $M_1\geq 1$. Introduce three domains
$$
\Omega_0=\{(x,y)\,:\,-\Lambda_j/2<X<\Lambda_j/2,\;0<Y<\xi(X)\}
$$
and
\begin{eqnarray*}
&&\Omega_1=\{(X,Y)\,:\,\Lambda_j/2<X<3\Lambda_j/2,\;0<Y<\xi(X)\},\\
&&\Omega_{-1}=\{(X,Y)\,:\,-3\Lambda_j/2<X<-\Lambda_j/2,\;0<Y<\xi(X)\}.
\end{eqnarray*}
Let $w_0(X,Y)=w(X,Y)$ for $X\in (-\Lambda_j/2,\Lambda_j/2)$ and $w_0(X,Y)=0$ for $X$ outside $(-\Lambda_j/2,\Lambda_j/2)$, where $w$ is the test function defined by (\ref{M31aa}). We put $w_{+1}(X,Y)=w_0(X-\Lambda_j,Y)$ and $w_{-1}(X,Y)=w_0(X+\Lambda_j,Y)$. Then the function
$$
v=w_0-\frac{1}{2}\Big(w_{+1}+w_{-1}\Big)\;\;\mbox{belongs to ${\mathcal X}_M$}.
$$
Then
$$
a_{M,t_j}(v,v)=2a_{1,t_j}(w,w)>0\;\;\mbox{for $w\neq 0$},
$$
where the last inequality follows from (\ref{A29b}). The reference to Proposition \ref{Pm26a}
completes the proof.
\end{proof}


\section{Main theorem}

\subsection{Bifurcation theorem for potential operators}

We will use the following bifurcation theorem for potential operator, see \cite{Ki1}, \cite{Ki2}.
Let $X$ and $Z$ be real Banach spaces and let $X$ be continuously embedded into $Z$. We assume that $Z$ is supplied with an inner product which is continuous with respect to the norm in $Z$. Consider a map
\begin{equation}\label{F17a}
F\,:\, U\times V\rightarrow Z,
\end{equation}
where $U$ is a neighborhood of $0$ in $X$ and $V$ is a neighborhood of $t_*$ in $\Bbb R$. We assume that

(i) $F\in C(U\times V,Z)$, $D_xF\in C(U\times V,L(X,Z))$ and $A(t):=D_xF(0,t)$ is a family of Fredholm operators of index zero such that
the operator $A(t)$ considered as unbounded operator in $Z$ with domain of definition $X$ is closed for each $t\in V$.

(ii) $F(\cdot,t)$ is a potential operator from $U$ into $Z$.

(iii) $F(0,t)=0$ for $s\in V$.

We note that (i) implies that the operator $A(t)$ is self-adjoint. We recall the definition of the crossing number. Let $0$ be an isolated eigenvalue of $A(t_*)$  and let $A(t)$ be invertible $t\in (t_*-\epsilon,t_*)\cup (t_*+\epsilon)$ for a small positive $\epsilon$. Denote by $\sigma_-(t)$ ($\sigma_+(t)$) the sum of multiplicities of perturbed eigenvalues of $A(t)$ near zero on the negative (positive) real axis.
Then the limit
$$
\chi (A(t),t_*):=\lim_{\epsilon\to 0}(\sigma_-(-\epsilon)-\sigma_+(\epsilon))
$$
exists and it is called the crossing number of the family $A(t)$ through $0$ at $t_*$.

\begin{theorem}\label{ThA5} (\cite{Ki1}, Theorem II.7.3) Let the family {\rm (\ref{F17a})} satisfy {\rm (i)--(iii)}. Let $0$ be an isolated eigenvalue of $A(t_*)$. If $A(t)$ is invertible for $V\setminus \{0\}$ and the crossing number of the family $A(t)$ at $t_*$ is nonzero then $(0,t_*)$ is a bifurcation point of $F(x,t)=0$ in the following sense: $(0,t_*)$ is a cluster point of nontrivial solution $(x,t)\in U\times V$, $x\neq 0$ of $F(x,t)=0$.

\end{theorem}

\subsection{Properties of the operator $S_M(t)$}\label{SA22ab}

Our aim is to apply Theorem \ref{ThA5} to the problem (\ref{A23a}), where we have the same differential operator and the same boundary conditions as in (\ref{A5aa}) but solutions have period $M\Lambda_0$. So we want to describe  bifurcation points for equation (\ref{A23a}), the kernel of the operator $S_M(t)$ when it is non-zero and study when the crossing number is non-trivial.

 The first step in the above programm is to describe bounded solutions to the problem
 \begin{eqnarray}\label{A6a}
 &&A(t)u=0\;\;\mbox{in $Q$},\nonumber\\
 &&N(t)u-u=0\;\;\mbox{for $p=1$},\nonumber\\
 &&u=0\;\;\mbox{for $p=0$}.
 \end{eqnarray}
 Usually this can be done by using Floquet exponents to the problem
\begin{eqnarray}\label{M21a}
&&A(\tau,t)w=0\;\;\mbox{in $Q$}\nonumber\\
&&N(\tau,t)w-w=0\;\;\mbox{for $p=1$}\nonumber\\
&&w=0\;\;\mbox{for $p=0$},
\end{eqnarray}
where
\begin{equation}\label{M23b}
A(\tau,t)w=e^{-i\tau q}A(t)(e^{i\tau q}w)
\end{equation}
and
\begin{equation}\label{M23ba}
N(\tau,t)w=e^{-i\tau q}N(t)(e^{i\tau q}w)
\end{equation}
Here $A(t)$ and $N(t)$ are given by (\ref{J4aa}) and (\ref{J4ab}) respectively. The spectral parameter $\tau \in\Bbb C$ for which the problem has non-trivial $\Lambda_0$-periodic solutions is called the Floquet exponent. Real $\tau$ produce bounded solution and complex $\tau$ corresponds to unbounded solutions at $+\infty$ or $-\infty$. As it is known (see \cite{ShaSo}, Sect.3.3) the $\tau$ spectrum of the problem (\ref{M21a}) consists of isolated Floquet exponents having finite algebraic multiplicity. We note that in study of problem (\ref{M21a}) we do not assume the corresponding eigenfunctions are even but certainly they have period $\Lambda_0$. Bounded solutions of (\ref{A6a}) are given by
$$
u(q,p)=e^{i\tau q}w(q,p),
$$
where $\tau$ is a real Floquet exponent and $w$ is a corresponding $\Lambda_0$-periodic eigenfunction. If we are looking for $M\Lambda_0$ periodic solutions then  $\tau M\Lambda_0=2k\pi$, i.e.
$$
\tau=\frac{2k\pi}{M\Lambda_0}=\frac{k}{M}\tau_*.
$$
Moreover,  solutions, even in $q$, are given by
$$
u_e(q,p)=e^{i\tau q}w(q,p)+e^{-i\tau q}w(-q,p).
$$
The above analysis implies, in particular, that if $M_1$ and $M_2$ are two prime numbers then the intersection of the kernels of $S_{M_1}$ and $S_{M_2}(t)$ coinsides with the kernel of the operator $S(t)$.

Another problem which is useful in study the crossing number of zero eigenvalue of the operator (\ref{M5ba}) is the following
\begin{eqnarray}\label{M21ad}
&&A(\tau,t)w=0\;\;\mbox{in $Q$},\nonumber\\
&&N(\tau,t)w-w=\mu w\;\;\mbox{for $p=1$},\nonumber\\
&&w=0\;\;\mbox{for $p=0$},
\end{eqnarray}
where $t\geq 0$ and $\tau$ is real.
Indeed, the study of crossing number is connected with eigenvalues of the operator $S_M(t)$ and this eigenvalues and corresponding eigenfunctions can be found with the help of the problem (\ref{M21ad}).

\subsection{Eigenvalues of (\ref{M21ad}) for $t=0$}

For $t=0$ we have $h=H(p)$, $\lambda=1$ and  the spectral problem (\ref{M21ad}) takes the form
\begin{eqnarray}\label{M14bb}
&&\Big(\frac{w_p}{H_p^3}\Big)_p+(\partial_q+i\tau)\Big(\frac{w_q+i\tau w}{H_p}\Big)=0\;\;\mbox{in $Q$}\nonumber\\
&&\frac{w_p}{H_p^3}-w=\mu w\;\;\mbox{for $p=1$}\nonumber\\
&&w=0\;\;\mbox{for $p=0$},
\end{eqnarray}
where $H$ is given by (\ref{M4c}).
To find solutions to this spectral problem we put
$$
w=\Gamma(q,H(p);\tau)H_p(p).
$$
Then according to Corollary \ref{CorM2} $\Gamma(x,y;\tau)$ satisfies
\begin{equation}\label{M14b}
((\partial_x+i\tau)^2+\partial_y^2)\Gamma+\omega'(H)\Gamma=0\;\;\mbox{for $0<y<d:=H(1)$}
\end{equation}
and
\begin{equation}\label{M14ba}
\Gamma(x,0;\tau)=0.
\end{equation}
Solutions to (\ref{M14b}), (\ref{M14ba})  with the period $\Lambda_0$ are given by
$$
\Gamma_n=e^{in\tau_* x}\gamma (y;\tau+n\tau_*),\;\;\mbox{$n$ is integer}
$$
where $\gamma$ is the solution to (\ref{Okt6b}).
 The  boundary condition at $p=1$ in (\ref{M14bb}) takes the form
$$
\frac{\gamma'(d,\tau+n\tau_*)}{H_p(1)}+\frac{H_{pp}(1)}{H_p^3(1)}-H_p(1)=H_p(1)\mu
$$
or
$$
\mu=\kappa^2\gamma'(d;\tau+n\tau_*)-1+\kappa\omega(1),
$$
where  the relation $H_p(1)=\kappa^{-1}$ is applied. Using the relation (\ref{Okt6ba}), we get
\begin{equation}\label{M21b}
\mu=\mu_n=\kappa\sigma(\tau+n\tau_*),\;\;n=0,\pm 1,\pm 2,\ldots
\end{equation}
Since $\sigma=\sigma(\tau)$ is an even function strongly increasing for $\tau>0$, we obtain

\begin{proposition}\label{PA20a} For every $\tau\in\Bbb R$ such that $2\tau\neq k\tau_*$, $k$ is an integer, the eigenvalues $\mu_n$ of the problem (\ref{M14bb}) are simple and given by (\ref{M21b}). Moreover, $\mu_n\neq 0$ for $\tau\in (0,\tau_*)$.
\end{proposition}

\begin{corollary}\label{CorA6} For every $\tau\in\Bbb R$ such that $2\tau\neq k\tau_*$ C the problem (\ref{M14bb}) has only simple eigenvalues except isolated points on the axis $t\geq 0$.
\end{corollary}

\subsection{Main result}

Let $\{\tilde{t}_j\}$, $j=1,\ldots,\infty$, be singular points of the operator $S(t)$, i.e. the points where the kernel operator $S(\tilde{t}_j)$ is non-trivial. As is known  the points $\tilde{t}_j$ are isolated (see \cite{BTT}), $\tilde{t}_j\to\infty$ as $j\to\infty$ and the operator $A(t)$ is invertible for $t\neq \tilde{t}_j$. We will use the numeration of these points such that
$$
0\leq \tilde{t}_1<\tilde{t}_2<\cdots
$$

\begin{theorem}\label{Tmain22} (i) There exists a sequence $(\widehat{t}_j, M_j)$, where $\widehat{t}_j\neq \tilde{t}_k$ for all $j$ and $k$; $M_j$ are prime numbers, and
$$
\widehat{t}_j\to\infty,\;\;M_j\to\infty\;\;\mbox{as $j\to\infty$}.
$$
Moreover, $\widehat{t}_j$ is a bifurcation point of the equation
\begin{equation}\label{A29ba}
{\mathcal S}_{M_{j}}(g;t)=0.
\end{equation}

(ii) There exists a sequence $(\widehat{t}_j, M_j)$, where $\widehat{t}_j\neq \tilde{t}_k$ for all $j$ and $k$; $M_j$ are prime numbers, the sequence $\{\widehat{t}_j\}$ is bounded and
$$
M_j\to\infty\;\;\mbox{as $j\to\infty$}.
$$
Furthermore, $\widehat{t}_j$ is a bifurcation point of the equation {\rm (\ref{A29ba})}.

According to {\rm Sect.\ref{SA22a}} the numbers $\widehat{t}_j$ are pairwise different in both cases.
\end{theorem}

\subsection{Proof} Let as before $\tilde{t}_j$, $j=1,\ldots,\infty$, be singular points for the operator function $S(t)$. From Theorem \ref{ThMa} it follows that $\tilde{t}_j\to\infty$ as $j\to\infty$.  By \cite{ShaSo} (see Proposition 3.4), the $\tau$ spectrum of the operators $S(t,\tau)$ consists of isolated  eigenvalues of finite algebraic multiplicities for every $t\geq 0$.

Let $M$ be a prime number. Consider the problem (\ref{A23a}) and the corresponding Frechet derivative $S_M(t)$ of the operator ${\mathcal S}_M(t)$. We use the splitting (\ref{A21a}) involving invariant subspaces of the self-adjoint operator $S_M(t)$ and the restriction operator $S_M(t)$ on ${\mathcal X}$, which was denoted by $\widehat{S}_M(t)$.
Introduce the singular points of the operator function $\widehat{S}_M(t)$ and denote them by $\widehat{t}_k$, $k=1,2,\ldots$.
Let us show that $\widehat{t}_k$ are isolated points. Indeed the operator  $\widehat{S}_M(t)$ is self-adjoint and its eigenvalues $\mu_j(t)$ depends analytically on $t$. If a certain $\mu_j$ has infinitely many zeros in a neighborhood of $\widehat{t}_k$ then it is identically zero for all $t$ which contradicts to Proposition \ref{PA20a}, which implies that $\mu_j(0)\neq 0$. 

Let $t_\dag >0$ and $t_\dag\neq \tilde{t}_j$, $j=1,\ldots,\infty$. Denote by ${\mathcal N}(t_\dag)$ the set of all prime numbers $M$ such that the $\tau$-spectrum of the operators $S(\tilde{t}_j)$ for all $\tilde{t}_j\leq t_\dag$ does not contain the points $k\tau_*/M$, $k=1,2,\ldots,M-1$. Then the
 kernel of the operator $\widehat{S}_M(\tilde{t}_j)$ for $\tilde{t}_j\leq t_\dag$ is trivial. 
 This means that
\begin{equation}\label{Okt15a}
\widehat{t}_k\neq \tilde{t}_j\;\;\mbox{for $\tilde{t}_j\leq t_\dag$, $\widehat{t}_k\leq t_\dag$ and $M\in {\mathcal N}(t_\dag)$ }.
\end{equation}

(ii) We choose $t_\dag$ and $M\in {\mathcal N}(t_\dag)$. 
By Proposition \ref{PA20a} all eigenvalues of the operators $\widehat{S}_M(t_\dag)$ are positive for small $t$. According to Proposition \ref{RA19a},
the operator $\widehat{S}_M(t_\dag)$ has negative eigenvalues  for sufficiently large $t_\dag$ lying near $t_j$ from Theorem \ref{ThMa}. By Proposition \ref{PrA29b}(ii) it can be chosen in such a way that all operators $\widehat{S}_M(t_\dag)$ has negative eigenvalues with arbitrary $M$.  This implies that for every $M\in \widehat{S}_M(t_\dag)$ there exists $\widehat{t}_k$ satisfying (\ref{Okt15a}) with the crossing number different from zero. We denote this number by $k(M)$. Now reference to  Theorem \ref{ThA5} proves existence of subharmonic bifurcatins with period $M\in{\mathcal N}(t_\dag)$ at the bifurcation points $\widehat{t}_{k(M)}$.

(i) Now let us choose $t_\dag^{(n)}$, $n=1,2,\ldots$, in the following way. First, they satisfy
$$
t_\dag^{(n)}\neq \tilde{t}_j\;\;\mbox{for all $n$ and $j$}\;\;0<t_\dag^{(1)}<t_\dag^{(2)}<\cdots\;\;t_\dag^{(n)}\to\infty\;\;\mbox{as $n\to\infty$}.
$$
Second, we assume that the number of negative eigenvalues of  the operators $\widehat{S}_M(t_\dag^{(1)})$ is positive. Applying  Proposition \ref{PrA29b}, we conclude that if $t_\dag^{(n-1)}$ is chosen then $t_\dag^{(n)}$ can be chosen to satisfy
\begin{equation}\label{Okt15b}
N_-(\widehat{S}_M(t_\dag^{(n-1)}))<N_-(\widehat{S}_M(t_\dag^{(n)})) \;\;\mbox{for all $M$},
\end{equation}
where $N_-$ denotes the number of negative eigenvalues.

Next we choose the prime numbers $M_n\in {\mathcal N}(t_\dag^{(n)})$. Due to (\ref{Okt15b}) there exist  there exists $\widehat{t}_k$, $k=k(M_n)$, satisfying 
$$
t_\dag^{(n-1)}<\widehat{t}_{k(M_n)}<t_\dag^{(n)}
$$
with the crossing number different from zero. The reference to  Theorem \ref{ThA5} proves existence of subharmonic bifurcations with period $M_n\in{\mathcal N}(t_\dag^{(n)})$ at the bifurcation point $\widehat{t}_{k(M_n)}$, $n=1,2,\ldots$.

\bigskip
\noindent
{\bf Acknowledgements} The author was supported by the Swedish Research Council (VR), 2017-03837.

\bigskip
\noindent
{\bf Data availability statement} Data sharing not applicable to this article as no datasets were generated or analysed during the current study.

\section{References}

{

\end{document}
\begin{thebibliography}{20}

\bibitem{AT1} C.  Amick and J. Toland, On periodic water-waves and their convergence to solitary waves in the long-wave
limit, Philos. Trans. Roy. Soc. London Ser. A, 303, 1981.


\bibitem{AT2}  CJ Amick, JF Toland, On solitary water-waves of finite amplitude
Arch. Ration. Mech. Anal., 76 (1), 1981.

 \bibitem{T2}   CJ Amick, LE Fraenkel, JF Toland, On the Stokes conjecture for the wave of extreme form,
Acta Mathematica 148 (1), 1982.

\bibitem{BDT1} B Buffoni, EN Dancer, JF Toland, The Regularity and Local Bifurcation of Steady Periodic Water Waves,
Archive for rational mechanics and analysis 152 (3), 207-240, 2000.

\bibitem{BDT2} B Buffoni, EN Dancer, JF Toland, The sub-harmonic bifurcation of Stokes waves,
Archive for rational mechanics and analysis 152 (3), 241-271, 2000.

\bibitem{BTT} B Buffoni, J Toland, Analytic theory of global bifurcation: an introduction, Princeton University Press, 2003.


\bibitem{Che} Chen, B. and Saffman, P.G. Numerical evidence for the existence of new types of gravity
waves on deep water. Stud. Appl. Math. 62, 1980.

\bibitem{CWW} RM Chen, S Walsh, MH Wheeler, Existence and qualitative theory for stratified solitary water waves,
 Annales de l'Institut Henri Poincaré C, 2018.

\bibitem{CSst} A Constantin, W Strauss, Exact steady periodic water waves with vorticity,
Communications on Pure and Applied Mathematics 57 (4), 481-527, 2004.

\bibitem{CSS} A Constantin, D Sattinger, W Strauss, Variational formulations for steady water waves with vorticity,
Journal of Fluid Mechanics 548, 151-163, 2006.

\bibitem{CSrVar} A Constantin, W Strauss, E Varvaruca, Global bifurcation of steady gravity water waves with critical layers,
Acta Mathematica 217 (2), 195-262, 2016.

\bibitem{C8} A. Constantin and W. Strauss, Periodic traveling gravity water waves with discontinuous vorticity,  Arch. Ration. Mech. Anal., 202 (2011).


\bibitem{Dan} T. M. Dunster, Bessel functions of purely imaginary order, with an application to second-order linear
differential equations having a large parameter, SIAM Journal on Mathematical Analysis, 21 (1990), pp. 995-
1018.

\bibitem{Gr} M.D. Groves, Steady water waves. J. Nonlin. Math.
Phys. 11 (2004).

\bibitem{GrW} Groves, M. D., Wahlen, E., Small-amplitude Stokes and solitary gravity water waves
with an arbitrary distribution of vorticity. Phys. D 237 (2008), no. 10–12.

\bibitem{Hur} Hur, Vera Mikyoung, Exact solitary water waves with vorticity. Arch. Ration. Mech.
Anal. 188 (2008), 2.

\bibitem{KNor} G. Keady and J. Norbury, On the existence theory for irrotational water waves, Math. Proc. Cambridge
Philos. Soc., 83 (1978).

\bibitem{Ki1} H Kielhöfer , Bifurcation theory: An introduction with applications to PDEs, Springer New York Dordrecht Heidelberg London, 2011.

\bibitem{Ki2} H Kielhöfer, A bifurcation theorem for potential operators, Journal of functional analysis, 1988.


\bibitem{KN14} V Kozlov, N Kuznetsov, Dispersion equation for water waves with vorticity and Stokes waves on flows with counter-currents,
Archive for Rational Mechanics and Analysis 214 (3), 971-1018, 2014.



\bibitem{KN11} V Kozlov, N Kuznetsov, Steady free-surface vortical flows parallel to the horizontal bottom,
The Quarterly Journal of Mechanics and Applied Mathematics 64 (3), 371-399, 2011.

\bibitem{KLN17} V Kozlov, N Kuznetsov, E Lokharu, On the Benjamin–Lighthill conjecture for water waves with vorticity,
Journal of Fluid Mechanics 825, 961-1001, 2017.





\bibitem{KL1} V Kozlov, E Lokharu, Global bifurcation and highest waves on water of finite depth,
arXiv preprint arXiv:2010.14156, 2020

\bibitem{KL2} V Kozlov, E Lokharu, On negative eigenvalues of the spectral problem for water waves of highest amplitude, JDE, Volume 342, 5 January 2023, Pages 239-281.

\bibitem{KL3} V Kozlov, E Lokharu, On Rotational Waves of Limit Amplitude,
Functional Analysis and Its Applications 55 (2), 165-169, 2021.

\bibitem{KL4} V Kozlov, E Lokharu, An asymptotic behaviour near the crest of waves of extreme form on water of finite depth,
arXiv preprint arXiv:2103.14451, 2021.

\bibitem{KLW} V Kozlov, E Lokharu, MH Wheeler, Nonexistence of subcritical solitary waves,
Archive for Rational Mechanics and Analysis, 241 (1), 535-552, 2021.

\bibitem{Kra} A. M. Krall, Boundary values for an eigenvalue problem with a singular potential, Journal of Differential
Equations, 45 (1982), pp. 128-138.




 \bibitem{P1} PI Plotnikov, Nonuniqueness of solutions of the problem of solitary waves and bifurcation of critical points of smooth functionals,
Math. USSR-Izv., 38 (2), 333, 1992.

\bibitem{P2 } PI Plotnikov, A proof of the Stokes conjecture in the theory of surface waves,
Studies in Applied Mathematics, 108 (2), 2002.

\bibitem{PT} P. I. Plotnikov, J. F. Toland, Convexity of Stokes
waves of extreme form. Arch. Ration. Mech. Anal. 171 (2004).



\bibitem{PrWe} Protter, M., Weinberger, H. : Maximum principles in differential equations. Prentice-Hall, 1967.


\bibitem{Rus} Scott Russell, J.: Report on Waves. Rep. 14th meet. Brit. Assos. Adv. Sci., John
Murray, London, 1844.



\bibitem{Sa} Saffman, P.G. Long wavelength bifurcation of gravity waves on deep water J. Fluid Mech.
101, 1980.

\bibitem{ShTo} E Shargorodsky, JF Toland, Bernoulli free-boundary problems, American Mathematical Soc., 2008.

\bibitem{ShaSo} E Shargorodsky, AV Sobolev, Quasiconformal mappings and periodic spectral problems in dimension two,
  Journal d'Analyse Mathématique, {\bf 91}, pp. 67–103, 2003.
	

\bibitem{Stokes} Stokes, G.G.: On the theory of oscillatory waves. Camb. Phil Soc. Trans. 8,
(1847).

\bibitem{Stokes2} G. G. Stokes, Considerations relative to the greatest height of oscillatory irrotational waves which can be propogated
without change of form, Mathematical and Physical Papers, 1 (1880).

\bibitem{Str} W. A. Strauss, Steady water waves, Bull. Amer. Math.
Soc. 47 (2010).

\bibitem{STrWh} W. Strauss, M. Wheeler, Bound on the slope of steady water waves with favorable vorticity, Archive for Rational Mechanics and Analysis, 2016.

\bibitem{T1a} J. Toland, On the existence of a wave of greatest height and Stokes’s conjecture, Proceedings of the RoyalSociety of London. A. Mathematical and Physical Sciences, 363 (1978).

    \bibitem{Van2}  Vanden-Broeck, J.-M. On periodic and solitary pure gravity waves in water of infinite
depth, J. Eng. Math. 84, 2013.

\bibitem{Van1}   Vanden-Broeck, J.-M. New families of pure gravity waves in water of infinite depth, Wave
Motion, 72, 2017.

\bibitem{VW1} E Varvaruca, GS Weiss, A geometric approach to generalized Stokes conjectures,
Acta mathematica 206 (2), 363-403, 2011.

\bibitem{VW2} E Varvaruca, GS Weiss, The Stokes conjecture for waves with vorticity,
Annales de l'Institut Henri Poincaré C, Analyse non linéaire 29 (6), 861-885, 2012.

\bibitem{Varv} E Varvaruca, On the existence of extreme waves and the Stokes conjecture with vorticity,
Journal of Differential Equations 246 (10), 4043-4076, 2009.


\bibitem{We} M. Wheeler, The Froude number for solitary water waves with vorticity,
Journal of Fluid Mechanics 768, 91-112, 2015.


























\end{thebibliography}
